\documentclass[12pt]{article}
\usepackage{amssymb,amsfonts,amsmath,amsthm, breqn, color,float}
\usepackage{caption}
\usepackage{pdfpages}
\usepackage{pgfplots}
\pgfplotsset{compat=1.14}
\usepgfplotslibrary{polar}
\usepgflibrary{shapes.geometric}
\usetikzlibrary{calc}
\usepackage{tikz, sidecap}
\usepackage[normalem]{ulem}
\usepackage{graphicx}
\usepackage{enumerate}
\pgfplotsset{my style/.append style={axis x line=middle, axis y line=middle, xlabel={$X$}, ylabel={$Y$}, axis equal }}
\newcommand{\ol}{\overline}
\newcommand{\one}[1]{\mbox {\bf 1}_{\{#1\}}}

\newcommand{\witi}{\widetilde}

\newcommand{\zz}{{\mathbb Z}}
\newcommand{\nn}{{\mathbb N}}

\newcommand{\rr}{{\mathbb R}}

\newcommand{\cale}{{\mathcal E}}

\newcommand{\calf}{{\mathcal F}}

\newcommand{\mz}{{\mathfrak Z}}
\newcommand{\my}{{\mathfrak Y}}
\newcommand{\veps}{\varepsilon}

\newcommand{\ycz}{Y^{(c,x,z)}}
\newcommand{\ycx}{Y^{(c,x,x)}}
\newcommand{\ucx}{U^{(c,x)}}
\newcommand{\ycm}{Y^{(c,x,z-x)}}
\newcommand{\uc}{U^{(c,n)}}
\newcommand{\Bc}{B^{(c,x)}}
\newcommand{\xc}{X^{(c,n)}}
\newcommand{\zc}{Z^{(c,n)}}
\newcommand{\qc}{Q^{(c,n)}}
\newcommand{\Sc}{S^{(c)}}

\newcommand{\xn}{X^{(n)}}

\newcommand{\beq}{\begin{eqnarray*}}
\newcommand{\feq}{\end{eqnarray*}}
\newcommand{\beqn}{\begin{eqnarray}}
\newcommand{\feqn}{\end{eqnarray}}

\newtheorem{theorem}{Theorem}
\makeatletter \@addtoreset{theorem}{section}\makeatother

\newtheorem{lemma}[theorem]{Lemma}
\newtheorem{assume}[theorem]{Assumption}
\newtheorem{cond}[theorem]{Condition}
\newtheorem*{theorema*}{Theorem~A}
\newtheorem*{theoremb*}{Theorem~B}
\newtheorem*{cld*}{Condition $\mbox{LD}_d$}
\newtheorem*{theorem*}{Theorem}
\newtheorem{proposition}[theorem]{Proposition}
\newtheorem{corollary}[theorem]{Corollary}
\newtheorem{remark}[theorem]{Remark}

\title{Avalanches in an excitable network}

\author{
Reza~Rastegar\thanks{Occidental Petroleum Corporation, Houston, TX 77046, USA; e-mail:  reza\_rastegar2@oxy.com}
\and
Alexander~Roitershtein \thanks{Dept. of Statistics, Texas A\&M University, College Station, TX 77843; e-mail: alexander@stat.tamu.edu}
}

\begin{document}
\maketitle

\begin{abstract}
We study propagation of avalanches in a certain excitable network. The model is a particular case of the one introduced in \cite{ca},
and is mathematically equivalent to an endemic variation of the Reed-Frost epidemic model introduced in \cite{longini}.
Two types of heuristic approximation are frequently used for models of this type in applications,
a branching process for avalanches of a small size at the beginning of the process and a deterministic
dynamical system once the avalanche spreads to a significant fraction of a large network. In this paper we prove several results
concerning the exact relation between the avalanche model and these limits, including rates of convergence and rigorous bounds for
common characteristics of the model.

\end{abstract}
\section{Introduction}
We study a discrete-time Markov model of the propagation of avalanches in a large network.
“Avalanches” here is a general term referring to a cascading spread of a node's feature in a network of linked objects.
The exact nature of the feature is immaterial for our purposes, it may have either positive or
negative effect on particular aspects of the network performance. We generally refer to network's nodes possessing the feature as excited.
Initially, a set of nodes becomes excited as a result of an external simulation. Once an avalanche is triggered, 
excited nodes can transmit this feature to currently non-excited ones, creating cascades (generations) of excited nodes evolving in discrete time.
\par
Examples of avalanches in networks that have been studied in applications include epidemics, outages in a power grid,
information spread in a human network, cascades of firing neurons in cortex, viruses in a computer network, forest fire etc \cite{gdurrett, ca}.
The model that we consider in this paper belongs to the class of chain-binomial Markov models \cite{infect, chainb}.
Two types of approximation are frequently used for models of this type, the first one is a branching process approximation for cascades
of a small size at the beginning of the process \cite{Andersson, bina, jmb, dgbook, gdurrett, n, a}, and the second one is an approximation by
a deterministic dynamical system once the avalanche spreads to a significant fraction of a large network \cite{ozgur, bdet, knerman, longini}.
\par
The approximations link the asymptotic behavior of immensely complex stochastic processes on a large network
to relatively well understood mathematical objects, allowing to gain a qualitative insight into statistical proprieties of the network model.
Each of the two approximation processes is in a rigorous sense a limit of the original model in a certain regime. The interpretation
of the link between the original model and the approximations is not trivial because some essential features are not preserved
when the limit is taken. For instance, the branching approximation for the model studied in this paper 
is in essence a linearization eliminating the dependence between the nodes (cf. \cite{ca}),
it is monotone in all basic parameters while the original model is not. Likewise, while the Markov chain describing the evolution of the 
avalanche magnitude (its transition kernel is specified in \eqref{ker} below) converges to zero with probability one for any set of parameters, 
there is a regime in which the approximating dynamical system converges to its non-zero global stable point (see Section~\ref{dapprox} below). 
Usually, the relation between network models and their approximations is studied using either heuristic arguments or numeric computations. 
\par 
In this paper we focus on a rigorous analytical comparison between the avalanche model
and the above mentioned limits, including rates of convergence and rigorous bounds for common characteristics of the model.
Loosely speaking, some of our results can be viewed as a “second order”
correction to the branching approximation. In order to make the comparison between the avalanche model and its branching approximation
we use coupling constructions based on canonical schemes of stochastic coupling of binomial and Poisson random variables
(see Sections~\ref{bpapr} and \ref{dtime} below). Most of our results are new, some complement the results obtained in \cite{ca} 
through heuristic perturbation arguments.   
\par
Typically, cascading models exhibit a phase transition between a subcritical regime characterized by a
short duration and small size  of the avalanche and a supercritical one characterized by long lasting avalanches that
eventually affect a non-zero fraction of the network before disappear. It is  often argued that regimes near the criticality exhibit
the most rich and advantageous for the network performance behavior. See, for instance, recent surveys \cite{review1, review} and references therein.
The phenomenon of criticality is of a special interest also because of a universal nature of the phenomenon as well
as mathematical challenges in its study. For an interesting discussion of the relation between branching processes and self-organized criticality see 
\cite{kello, b10}. The model that we investigate in this paper is lacking a trivial monotonicity in parameters,
but nevertheless exhibits the phase transition with a distinct critical set of parameters.
\par
We proceed with a formal definition of the avalanche model considered in this paper. 
Fix an integer $n\geq 3$ and denote $V_n=\{1,\ldots,n\}.$
The set $V_n$ models the nodes of a network with $n$ nodes. Let $\Omega_n:=
2^{V_n}$ be the space of subsets of $V_n,$ and consider the following \emph{avalanche process} $(A_k)_{k\in\zz_+}$ on $\Omega_n.$
Here and henceforth $\zz_+$ denotes the set of non-negative integers. Assume that the initial state $A_0\in\Omega_n\backslash\{\emptyset,V_n\}$ is neither
an empty set nor the whole network. Let $p\in (0,1)$ be given, and $q=1-p.$
Formally, the sequence $A_k$ is a discrete-time Markov chain in the state space $\Omega_n$ with transition kernel given by
\beqn
\label{amset}
P(A_{k+1}=B\,|\,A_k=A)=
\left\{
\begin{array}{ll}
\bigl(1-q^{|A|}\bigr)^{|B|}\cdot \bigl(q^{|A|}\bigr)^{n-|A|-|B|}&\mbox{\rm if}~B\subset A^c\\
0&\mbox{\rm otherwise},
\end{array}
\right.
\feqn
where $A$ and $B$ are arbitrary elements of $\Omega_n,$ $A^c$ denotes the complement of the set $A$ in $V_n,$ and $|A|$ denotes the cardinality of $A.$
\par
We refer to $A_k$ and $A_k^c$ as, respectively, the \emph{excited} and \emph{resting}
states at time $k.$ The interpretation is that an excited state turns into rested in the next instance of time,
but before that can excite any of the resting states, each one with probability $p.$
Thus the (conditional, given the excited nodes $A_k$) probability that a node
$x\in A_k^c$ will \emph{not} become excited in the next iteration is equal to $q^{|A_k|}.$
We further assume that the excitement mechanisms of different resting nodes
are independent each of other at any given instant of time, and hence the product on the right-hand side of \eqref{amset}.
\par 
The model is a particular case of the avalanche model introduced in \cite{ca},
where the ``excitant probability" is $p,$ uniformly across all links in the network. Formally, the model described by \eqref{amset}
coincides with the model of the spread of an endemic infection introduced in \cite{longini}. In contrast to \cite{longini},
we concentrate in this paper on the case when $p$ is $O(1/n)$ rather than $O(1)$ for large $n.$ More realistic versions 
of excitable networks are considered, for instance, in \cite{b7} and \cite{b6, iva}. We remark that our proof methods can be partially 
extended and applied to more complex networks, cf. \cite{reza15}.  
\par
Let $\{G_k(n,p):k\in\zz_+\}$ be an i.\,i.\,d. sequence of Erd\H{o}s-R\'{e}nyi graphs with percolation parameter $p,$
sharing $V_n$ as the common vertex set. An equivalent, dynamic graph viewpoint on the avalanche model is that every node excited at time $k$
excites with certainty all its resting neighbors in the random graph $G_k(n,p).$ Though this observation is not used anywhere in the paper,
it immediately provides some heuristic insights into the behavior of the avalanche model when $n$ is large. For instance, if $p=c/n$ then
with overwhelming probability, when $c\in(0,1)$ all connected components of the Erd\H{o}s-R\'{e}nyi graph are of order $O(\log n),$ while
if $c>1$ there is a connected component of the size $O(n)$ \cite{rgraphs}. Heuristically, an implication for the avalanche model that one may expect, is that
the avalanche starting on a single node has a little chance to spread over a non-zero fraction of the network in the former case
whereas the probability of such an avalanche to eventually reach the size of order $O(n)$ is non-zero in the latter regime. This heuristic observation
is formally confirmed in Section~\ref{fixp}. More generally, the dynamic graph viewpoint might 1) serve as an indication of the existence of a phase transition
in terms of the asymptotic behavior of the model on the scale $p=c/n$ at $c=1,$ and 2) be perceived as a fundamental reason behind the phase transition.
\par
In this paper we focus on the Markov chain $(X_t)_{t\in\zz_+},$ where $X_t=|A_t|.$ Transition kernel of this Markov chain is as follows:
\beqn
\label{ker}
P(X_{k+1}=j|X_k=i)=\binom{n-i}{j} (1-q^i)^j(q^i)^{n-i-j},\qquad i=0,1,\ldots,n,\,\quad j\leq n-i,
\feqn
with the convention that $\binom{0}{0}=1$ and $\binom{0}{j}=0$ for all $i\in\nn.$
\par
Let $T$ and $S$ denote, respectively, the duration and the total size of the avalanche:
\beqn
\label{dusize}
T=\inf\{k\in\zz_+: X_k=0\}\qquad \mbox{\rm and}\qquad S=\sum_{k=0}^T X_k,
\feqn
with the usual convention that $\inf \emptyset =+\infty.$ Remark that zero is the unique absorbing state and $P(T<\infty)=1$ regardless
of the initial state $X_0.$  Both the quantities are arguably among the most important general characteristics of avalanches in a complex network. 
In Section~\ref{dtime} we study the distribution function of avalanche duration $T.$ The asymptotic behavior of $S$ for large
values of $n$ is the content of Theorem~\ref{nsumse} in Section~\ref{bpapr}.
\par
The organization of the paper is as follows. In Section~\ref{funda} we obtain first estimates for the expected value of
several fundamental characteristics of the avalanches, such as the current size $X_k,$ current heterogeneity $X_k(n-X_k),$
and the total size $S.$ Some of these basic estimates are subsequently refined and improved. In Section~\ref{bpapr} we introduce the key
technical tool of our study, a branching process to which the avalanche chain converges weakly as $n$ tends to infinity, provided that
the initial value $X_0$ is kept fixed and the ``intensity factor" $pn$ converges to a positive limit. We use the branching approximation
to study the total size of the avalanche in the subcritical regime. In particular, Theorem~\ref{nsumse} gives the rate of convergence
to the limiting value as $n$ goes to infinity. In Section~\ref{fixp} we address the question whether an avalanche that started on a few
initially excited nodes can propagate to a non-zero fraction of the network (asymptotically, when $n$ is large). In particular,
Theorem~\ref{main1} provides lower and upper bounds with a qualitatively matching  asymptotic behavior for this probability for a given
network size $n.$ In Section~\ref{dtime} we study duration of the avalanche using branching approximation and its natural coupling
with the avalanche chain. In Section~\ref{dapprox} we are concerned with an approximation of the avalanche chain by a deterministic dynamical system.
While the branching approximation is adequate as long as $X_k\ll n,$ the deterministic approximation is suitable when $X_k/n=O(1).$
\section{Basic estimates for a subcritical network}
\label{funda}
The aim of this section is to give bounds on the expected number of excited nodes at a given time, total size of the avalanche,
and the probability that a given node is excited at a fixed time $k\in\nn.$ Throughout the section we consider a single network,
that is we assume that $n\in\nn$ and $p\in (0,1)$ in \eqref{ker} are fixed. Most of the results here are of an auxiliary nature, but some,
in particular Propositions~\ref{heter}, \ref{mixing}, and~\ref{expo}, appear to be of independent interest.
Propositions~\ref{heter} is concerned with the evolution of a measure of heterogeneity for the network,
Proposition~\ref{mixing} gives a uniform on $V_n$ upper bound on the probability that a given node $x\in V_n$ is excited at time $k,$ and
Proposition~\ref{expo} gives tight lower and upper bounds for the expected total size of the avalanche in the subcritical regime.
For future convenience, we formulate our results in terms of arbitrary bounds $c>0$ and $d>0$ that satisfy the following condition:
\begin{cond}
\label{assume5}
$0<d\leq pn \leq c.$
\end{cond}
The next series of propositions is formulated for an arbitrary $c>0,$ even though the results are primarily useful in the case when $c\in (0,1).$
As we will see in the next section, this case exactly corresponds to a subcritical regime of the avalanche model, which turns out to be
$p$ satisfying the condition $pn<1.$
\par
Let $(\calf_k)_{k\in\zz_+}$ be the natural filtration for the sequence $(X_k)_{k\in\zz_+},$ that is $\calf_k$ is the
$\sigma$-algebra generated by $X_0,\ldots,X_k.$ It follows from \eqref{ker} that with probability one,
\beqn
\label{avee}
E(X_{k+1}\,|\,\calf_k)\leq (n-X_k)\Bigl[1-\Bigl(1-\frac{c}{n}\Bigr)^{X_k}\Bigr]\leq cX_k\frac{n-X_k}{n}\leq cX_k.
\feqn
Notice that the formula remains formally true when $X_k=0.$ Thus, we have:
\begin{proposition}
\label{ave}
Let Condition~\ref{assume5} hold. Then the sequence $(X_kc^{-k})_{k\in \zz_+}$ is a supermartingale with respect to the filtration $(\calf_k)_{k\in\zz_+}.$
\end{proposition}
This straightforward estimate will be improved for the entire range of the network parameters ($np$ is less, greater, or equal to $1$)
in Section~\ref{dapprox} below.
\par
For $k\in\zz_+,$ let
\beqn
\label{het}
H_k=X_k(n-X_k).
\feqn
The following is immediate from \eqref{avee}.
\begin{corollary}
\label{avec}
Let Condition~\ref{assume5} hold. Then $E(X_k)\leq \frac{c^k}{n}E(H_0)\leq \frac{c^k}{n}E(X_0)$ for all $k\in\nn.$
\end{corollary}
For $k\in\zz_+$ and $x\in V_n,$ let
\beqn
\label{wk}
\cale_k(x)=\left\{
\begin{array}{ll}
1&\mbox{\rm if}~x\in A_k\\
0&\mbox{\rm if}~x\not\in A_k
\end{array}
\right.
\feqn
be the indicator of the event that a given node $x\in V_n$ is excited at time $k\in\zz_+.$ Remark that $H_k=X_k(n-X_k)$ can be interpreted as
a (non-normalized) measure of the heterogeneity of the network because
\beq
H_k=X_k(n-X_k)=\Bigl(\sum_{x \in V_n} \cale_k(x)\Bigr)\cdot \Bigl(\sum_{x \in V_n}\bigl(1-\cale_k(x)\bigr)\Bigr)
=\sum_{x,y \in V_n} \cale_k(x)\bigl(1-\cale_k(y)\bigr),
\feq
and thus $\frac{2H_k}{n(n-1)}$ is the probability that two nodes randomly chosen at time $k$ are not in the same state. It turns
out (see Section~\ref{dapprox} in this paper) that in the subcritical case $np\in(0,1)$ (and in fact also in the critical case $np=1$), for large values of $n,$ loosely speaking,
the number of excited nodes decreases to zero almost monotonically.
This observation motivates the following result, which in particular implies that the expected heterogeneity decreases to zero monotonically
in the subcritical regime (see also Corollary~\ref{h} below).
\begin{proposition}
\label{heter}
Let Condition~\ref{assume5} hold. Then the sequence $(H_kc^{-k})_{k\in\zz_+}$ is a supermartingale with respect to the filtration $(\calf_k)_{k\in\zz_+}.$
\end{proposition}
\begin{proof}
It follows from \eqref{ker} that
\beqn
\label{sm}
E(X_{k+1}^2|X_k)&=&(n-X_k)(1-q^{X_k})q^{X_k}+(n-X_k)^2(1-q^{X_k})^2.
\feqn
Notice that the formula remains true when $X_k=0.$ It follows that
\beqn
\nonumber
&&
E\bigl[X_{k+1}(n-X_{k+1})\bigr|X_k\bigr]=nE(X_{k+1}|X_k)-E(X_{k+1}^2|X_k)
\\
\nonumber
&&
\qquad
=n(n-X_k)(1-q^{X_k})- (n-X_k)(1-q^{X_k})q^{X_k}-(n-X_k)^2(1-q^{X_k})^2
\\
\nonumber
&&
\qquad
=
(n-X_k)(1-q^{X_k})\bigl[n-q^{X_k}-(n-X_k)(1-q^{X_k})\bigr]
\\
\nonumber
&&
\qquad
=
(n-X_k)(1-q^{X_k})\bigl[X_k(1-q^{X_k})+(n-1)q^{X_k}\bigr]
\\
\nonumber
&&
\qquad
\leq
(n-1)(n-X_k)(1-q^{X_k})
\leq
n(n-X_k)\Bigl(1-\Bigl(1-\frac{c}{n}\Bigr)^{X_k}\Bigr)
\\
\label{hier}
&&
\qquad
\leq
cX_k(n-X_k),
\feqn
where in the first inequality we used the fact that $X_k\leq n-1,$ and hence
\beq
X_k(1-q^{X_k})+(n-1)q^{X_k}\leq n-1.
\feq
The proof of the proposition is complete.
\end{proof}
Recall $\cale_k(x)$ from \eqref{wk}. Let
\beqn
\label{xik}
\xi_k(x)= P\bigl(\cale_k(x)=1\bigr), \qquad x\in V_n
\feqn
and
\beq
\eta_k=P(X_k>0)=P(T>k), \qquad k \in\zz_+,
\feq
where $T$ is the duration of the avalanche defined in \eqref{dusize}. We have the following:
\begin{proposition}
\label{mixing}
Let Condition~\ref{assume5} hold. Then $\xi_k(x)\leq  \frac{c^k}{n^2}E(H_0)$ for all $k\in\nn$ and $x\in V_n.$
\end{proposition}
The proof of the proposition is given below in this section, after the statement of a corollary.
The claim is trivial if the distribution of $X_0$ is invariant with respect to permutation of nodes (see the proof below),
but takes slightly more effort to establish when the symmetry is broken and the nodes cannot be treated as stochastically identical.
\par
We now proceed with the corollary. By virtue of \eqref{ker},
\beq
P(X_{k+1}=0\,|\,X_k)\geq q^{\frac{n^2}{4}}
\feq
uniformly on $X_k,$ and hence
\beqn
\label{ek}
\eta_k=P(X_k>0)<(1-q^{\frac{n^2}{4}})^k\qquad \forall\,k\in\nn.
\feqn
The following result is a refinement of this naive estimate for the subcritical regime.  By Chebyshev's inequality,
\beq
\eta_k=P(X_k\geq 1)\leq E(X_k)=E\Bigl(\sum_{x\in V_n}\cale_k(x)\Bigr)=\sum_{x\in V_n}\xi_k(x).
\feq
This yields
\begin{corollary}
\label{mixingc}
Let Condition~\ref{assume5} hold. Then $\eta_k \leq  \frac{c^k}{n}E(H_0)$ for all $k\in\nn.$
\end{corollary}
Remark that the result in the corollary will be further improved in Theorem~\ref{main7} below using a comparison to a branching process and known estimates for
the extinction time of the latter. In particular, it turns out that if $np=1$ in \eqref{ker}, then
\beq
\eta_k \leq  1-\bigl(\frac{k}{k+2}\bigr)^{i_0}\leq \frac{2i_0}{k+2}\qquad  \forall\,k\in\nn,
\feq
and the bound is asymptotically tight (see the lower bound in Theorem~\ref{main7} and also Corollary~\ref{th7}).
\begin{proof}[Proof of Proposition~\ref{mixing}]
We say that the distribution of $A_0$ is \emph{exchangeable} if it is independent of a particular labeling of the network's nodes, that if
$P(B\subset A_0)=P(\sigma B\subset A_0)$ for any permutation (one-to-one and onto relabeling) $\sigma:V_n\to V_n$ and a cluster of nodes $B\subset V_n.$
\par
Fix any $x\in V_n.$ First, observe that if the distribution of $A_0$ is exchangeable, then
\beq
\xi_k(x)=E\bigl(\cale_k(x)\bigr)=\frac{1}{n}E\Bigl(\sum_{x\in V_n} \cale_k(x)\Bigr)=\frac{1}{n}E(X_k)\leq \frac{c^k}{n}E(H_0)
\feq
by virtue of Proposition~\ref{ave}.
\par
Consider now an auxiliary avalanche process where $A_0$ is chosen uniformly over subsets of $V_n$ of a given size $m.$
We will denote the conditional law of the process $P(\,\cdot\,|A_0)$ by $P_{m,x}$ when $x\in A_0$ and by $P_{m,\ol x}$ otherwise.
Then, in view of the result for exchangeable $A_0$ and Proposition~\ref{mixing},
\beq
\frac{m(n-m)c^k}{n^2}&=&\frac{c^k}{n^2}E(H_0)\geq P\bigl(\cale_k(x)=1\bigr)
\\
&=&\frac{\binom{n-1}{m}}{\binom{n}{m}} P_{m,\ol x}\bigl(\cale_k(x)=1\bigr)+\frac{\binom{n-1}{m-1}}{\binom{n}{m}} P_{m,x}\bigl(\cale_k(x)=1\bigr)
\\
\label{e1}
&\geq&
\frac{\binom{n-1}{m}}{\binom{n}{m}} P_{m,\ol x}\bigl(\cale_k(x)=1\bigr)=\frac{n-m}{n}P_{m,\ol x}\bigl(\cale_k(x)=1\bigr),
\feq
which implies
\beqn
\label{ol}
P_{m,\ol x}\bigl(\cale_k(x)=1\bigr)\leq \frac{mc^k}{n}.
\feqn
Therefore,
\beqn
\nonumber
P_{j,x}\bigl(\cale_{k+1}(x)=1\bigr)&=&\sum_{m=0}^{n-j}\binom{n-j}{m}(1-q^j)^mq^{j(n-j-m)}P_{m,\ol x}\bigl(\cale_k(x)=1\bigr)
\\
\nonumber
&\leq& c^k \sum_{m=0}^{n-j}\binom{n-j}{m}(1-q^j)^mq^{j(n-j-m)}\frac{m}{n}
\\
\label{ole}
&=&
\frac{c^k}{n} E_{j,x}(X_1)=
\frac{c^k(n-j)(1-q^j)}{n} \leq \frac{c^{k+1}j(n-j)}{n^2}.
\feqn
In particular, we have established that
\beqn
\label{ol1}
P_{j,x}\bigl(\cale_k(x)=1\bigr)\leq \frac{c^{k-1}(n-j)(1-q^j)}{n}\leq c^k\frac{j(n-j)}{n^2} \leq \frac{jc^k}{n}.
\feqn
Turning now to $P_{j,\ol x}\bigl(\cale_{k+1}(x),$ write
\beq
P_{j,\ol x}\bigl(\cale_{k+1}(x)=1\bigr)&=&\sum_{m=1}^{n-j-1}\binom{n-j-1}{m}(1-q^j)^mq^{j(n-j-m-1)}P_{m,\ol x}\bigl(\cale_k(x)=1\bigr)
\\
&&
\qquad
+ \sum_{m=1}^{n-j-1}\binom{n-j-1}{m-1}(1-q^j)^{m-1}q^{j(n-j-m)}P_{m,x}\bigl(\cale_k(x)=1\bigr).
\feq 
Using \eqref{ol1} along with \eqref{ol} , we obtain
\beqn 
\nonumber
P_{j,\ol x}\bigl(\cale_{k+1}(x) &\leq&
\frac{c^k}{n}\sum_{m=1}^{n-j-1}\binom{n-j-1}{m}(1-q^j)^mq^{j(n-j-m-1)}m
\\
\nonumber
&&
\qquad
+ \frac{c^k}{n}\sum_{m=1}^{n-j-1}\binom{n-j-1}{m-1}(1-q^j)^{m-1}q^{j(n-j-m)}m
\\
\label{ole3}
&=&
\frac{c^k}{n} E_{j,\ol x}(X_1)=
\frac{c^k(n-j)(1-q^j)}{n}\leq \frac{c^{k+1}j(n-j)}{n^2}.
\feqn
The claim follows now from \eqref{ole} and \eqref{ole3}.
\end{proof}
We next investigate the total number of excited nodes in a subcritical regime. We will use the following lower bound for
$1-q^i.$
\begin{lemma}
\label{loq}
Under Condition~\ref{assume5}, $1-q^i\geq  \frac{di}{n}-\frac{c^2i^2}{2n^2}$ for all $i\in\zz_+.$
\end{lemma}
\begin{proof}
The lemma is trivial for $i=0,1.$ For $i\geq 2,$ using the Lagrange form of the second order remainder in Taylor's series for $f(p)=(1-p)^i$ around zero,
\beq
(1-p)^i=1-ip+\frac{i(i-1)}{2}p^2(1-p_*)^{i-2}\leq1-ip+\frac{i(i-1)}{2}p^2\leq 1- \frac{di}{n}+\frac{c^2i^2}{2n^2}
\feq
for some $p_*\in (0,p).$
\end{proof}
Recall $S$ from \eqref{dusize}. We have:
\begin{proposition}
\label{expo}
Suppose that Condition~\ref{assume5} holds with $c\in (0,1).$ Then,
\beqn
\label{sumse}
\frac{E(X_0)}{1-d}-\frac{3E(X_0^2)}{n(1-c)^3}
\leq E(S) \leq \frac{E(X_0)}{1-c}.
\feqn
\end{proposition}
\begin{proof}
To prove the proposition, we will first obtain a suitable lower bound for $E(X_k).$ It follows from Lemma~\ref{loq} that
\beqn
\nonumber
E(X_{k+1}|X_k)&\geq& (n-X_k)\Bigl(\frac{dX_k}{n}-\frac{c^2X_k^2}{2n^2}\Bigr)=
dX_k-\frac{c^2X_k^2}{2n}-\frac{dX_k^2}{n}+\frac{c^2X_k^3}{2n^2}
\\
\label{sume}
&\geq&
dX_k-\frac{3cX_k^2}{2n},
\feqn
where we used the fact that $d<c$ and $c^2<c.$ Using the identity in \eqref{sm}, we obtain
\beqn
\nonumber
E(X_k^2)&=&
E\bigl[(n-X_{k-1})(1-q^{X_{k-1}})q^{X_{k-1}}+n^2(1-q^{X_{k-1}})^2\bigr]
\\
\nonumber
&\leq& E\bigl[n(1-q^{X_{k-1}})+n^2(1-q^{X_{k-1}})^2\bigr]
\\
\label{sqe}
&\leq& E\Bigl[n\cdot \frac{c X_{k-1}}{n}+n^2 \frac{c^2 X_{k-1}^2}{n^2}\Bigr]=
E\bigl(c X_{k-1}+c^2 X_{k-1}^2\bigr).
\feqn
Iterating,
\beq
E(X_k^2)&\leq&E\bigl(c X_{k-1}+c^2 X_{k-1}^2\bigr)
\leq  E\bigl(c X_{k-1}+c^3 X_{k-2} +c^4X_{k-2}^2\bigr)
\\
&\leq & c^{2k}E(X_0^2)+\sum_{j=1}^k c^{2j-1}E(X_{k-j})
\leq  c^{2k}E(X_0^2)+\sum_{j=1}^k c^{k+j-1}E(X_0)
\\
&\leq & c^{2k}E(X_0^2)+\frac{c^k}{1-c}E(X_0).
\feq
Therefore,
\beq
E(X_{k+1})\geq
dE(X_k)-\frac{3c}{2n}E\Bigl[c^{2k}E(X_0^2)+\frac{c^k}{1-c}E(X_0)\Bigr].
\feq
Iterating again, we obtain
\beq
E(X_{k+1})&\geq& d^{k+1}E(X_0)-\frac{3c}{2n}\sum_{j=0}^k\Bigl[c^{2(k-j)}d^jE(X_0^2)+\frac{c^{k-j}d^j}{1-c}E(X_0) \Bigr]
\\
&\geq& d^{k+1}E(X_0)-\frac{3c}{2n}\sum_{j=0}^k\Bigl[c^{2k+j}E(X_0^2)+\frac{c^k}{1-c}E(X_0) \Bigr]
\\
&\geq&
d^{t+1}E(X_0)-\frac{3c^{k+1}}{2n(1-c)}\bigl[c^kE(X_0^2)+(k+1)E(X_0) \bigr].
\feq
Using this bound along with the upper bound in Corollary~\ref{avec}, we obtain that
\beq
d^kE(X_0)-\frac{3c^k}{2n(1-c)}\bigl[c^kE(X_0^2)+kE(X_0) \bigr]\leq E(X_k)\leq c^kE(X_0).
\feq\
Therefore, summing over all indexes from zero to $k,$
\beq
\frac{E(X_0)}{1-d}-\frac{3}{2n(1-c)}\Bigl[\frac{1}{1-c^2}E(X_0^2)+\frac{c}{(1-c)^2}E(X_0) \Bigr].
\leq E\Bigl(\sum_{k=0}^\infty X_k\Bigr) \leq \frac{E(X_0)}{1-c}
\feq
Taking in account that $c< 1,$ $1-c^2<(1-c)^2,$ and $E(X_0)\leq E(X_0^2),$
we obtain the lower bound in the form given in the statement of the proposition.
\end{proof}
We remark that though the constant $\frac{3E(X_0^2)}{(1-c)^3}$ in front of $1/n$ in the correction term at the left-hand side
of \eqref{sumse} is not optimal, the lower bound captures correctly the dependence of this term on $n.$
The latter result is formally stated in part (iii) of Theorem~\ref{nsumse} below.
\section{Poisson approximation and the size of the avalanche}
\label{bpapr}
The primary goal of this section is to study the asymptotic behavior of the total size of the avalanche for a
certain ensemble of comparable avalanche models. The underlying family of models is introduced in equation \eqref{ker1} and Assumption~\ref{assume7} below,
and the main result of this section is stated in Theorem~\ref{nsumse}. The secondary purpose of this section is to introduce
a branching process approximation which will be used throughout the rest of the paper.
\par
In the rest of the paper, along with a single network $X,$ we will often consider a family of Markov chains $\xn=(\xn_k)_{k\in\zz_+},$
each governed by a transition kernel of the same type as in \eqref{ker}, namely
\beqn
\label{ker1}
P_n(i,j):=P\bigl(\xn_{k+1}=j\bigl|\xn_k=i\bigr)=\binom{n-i}{j} (1-q_n^i)^j(q_n^i)^{n-i-j},
\feqn
for some $q_n\in (0,1)$ and all $i=0,1,\ldots,n,$ $j\leq n-i.$ In this definition we maintain the convention
that $\binom{0}{0}=1$ and $\binom{0}{j}=0$ for all $i\in\nn$ in \eqref{ker1}.
Therefore, all Markov chains in this collection eventually absorb at zero. Typically we will impose the following comparability assumption
on the family of avalanche models under consideration:
\begin{assume}
\label{assume7}
\item [(i)] There exists $\lambda>0$ such that $\lim_{n\to\infty} np_n=\lambda,$ where $p_n=1-q_n.$
\item [(ii)] All $\xn$ have the same initial state, namely $\xn_0=i_0$ for some $i_0\in\nn$ and all $n\in\nn.$
\end{assume}
Some of our asymptotic estimates will be stated in terms of arbitrary numbers $c>0,$ $d>0,$ and $i_0\in \nn$ that satisfying the following condition.
This condition is an analogue of Condition~\ref{assume5} for a family of networks that satisfies Assumption~\ref{assume7}.
\begin{assume}
\label{cdla}
Let $p_n\in(0,1),$ $n\in\nn,$ be given and consider a family of avalanche models $\{X^{(n)}:n\in\nn\}$ with
transition kernels defined in \eqref{ker1}. Assume that part (ii) of Assumption~\ref{assume7} is in force and, furthermore,
there exist constants $c>0,$ $d\in (0,c),$ and $n_0\in\nn$ such that
\item [(i)]  $np_n \in [c,d]$ for all $n\geq n_0.$
\item [(ii)] If part (i) of Assumption~\ref{assume7} holds and $\lambda>1,$ then $d>1.$
\item [(iii)] If part (i) of Assumption~\ref{assume7} holds and $\lambda<1,$ then $c<1.$
\end{assume}
It follows from \eqref{ker} that under Assumption~\ref{assume7}, for any $i\in\nn,$ $j\in\zz_+,$ and $k\in \nn,$
\beqn
\label{kernelc}
\lim_{n\to\infty} P\bigl(\xn_{k+1}=j\bigl|\xn_k=i\bigr)=e^{-\lambda i}\frac{(\lambda i)^j}{j!}.
\feqn
Let $Z^{(\lambda)}=\bigl(Z^{(\lambda)}_k\bigr)_{k\in \zz_+}$ be a Markov chain on $\zz_+$ with absorption state at zero, Poisson transition kernel
\beq
P\bigl(Z^{(\lambda)}_{k+1}=j\bigl|Z^{(\lambda)}_k=i\bigr)=e^{-\lambda i}\frac{(\lambda i)^j}{j!},\qquad i\in \nn,~j\in\zz_+,
\feq
and the same initial state $Z^{(\lambda)}_0=i_0,$ the same as for all $X^{(n)}.$
We can assume without loss of generality that $Z^{(\lambda)}$ is a Galton-Watson branching process with a Poisson offspring distribution, namely
\beqn
\label{br}
Z^{(\lambda)}_{k+1}=\sum_{j=1}^{Z^{(\lambda)}_k} Y^{(\lambda)}_{k,j}
\feqn
for a collection of independent Poisson random variables $Y=\{Y^{(\mu)}_{k,j}:k\in \zz_+,j\in\nn,\mu>0\}$ such that for all $i\in\zz_+,$
\beq
P\bigl(Y^{(\mu)}_{k,j}=i\bigr)=e^{-\mu}\frac{\mu^i}{i!}.
\feq
The sum in the right-hand side of  \eqref{br} is assumed to be zero if $Z^{(\lambda)}_k=0,$ that is $Z^{(\lambda)}_k$ is formally defined for all $k\in\zz_+.$
\par
The convergence in \eqref{kernelc} implies the weak convergence of the sequence of Markov processes  $\xn$ to the branching process 
$Z^{(\lambda)}$ as $n\to\infty$ \cite{karr}. To illustrate the functionality of the branching approximation, Fig.~\eqref{f100} and~\eqref{f103} 
below provide plots of $E(T)$ as a function of the initial state $i_0$ for $n=100$ and $n=1000,$ in each case for four values of the 
parameter $c=np$ concentrated around the theoretical phase transition value $c=1$ suggested by the branching approximation. 
Note that $E(T)<\infty$ by virtue of \eqref{ek}. 
\par
Let $Q$ be $(n-1)\times (n-1)$ matrix with entries $Q(i,j)=P(X_{k+1}=j\,|\,X_k=i).$
To evaluate the expectation we use the following standard Markov chain matrix calculation:
\beq
E(T)=\sum_{m=0}^\infty P(T>m)=\sum_{m=0}^\infty Q^m e=(I-Q)^{-1}e,
\feq
where $e\in\rr^{n-1}$ is an $(n-1)$-vector with all entries equal to one and $I$ is the $(n-1)$-dimensional unit matrix.
To compute the inverse matrix in the above expression we used the packages ``numpy" and ``decimal" on Python~3.5 with the computation
precision set to 400 decimal points.
\par
An intuitive reason for the uniformly (on $i_0$) large values of $E(T)$ and high persistence of the avalanche in the supercrtical regime, 
when $n$ is large, is that the stochastic path of the Markov chain $X$ is well approximated by a trajectory of a
deterministic dynamical system that is locally Lipschitz, and consequently is quickly attracted to its unique (non-zero) global stable point (see 
Section~\ref{dapprox} below for details). Heuristically, it appears that the Markov chain spends most of its 
time before the absorbtion being ``trapped" in a neighborhood of the stable point. 
The numerical simulations show that the phase transition in the avalanche model doesn't occur at exactly $c=1$ for either $n=100$ or $1000.$ 
While the phase transition is fairly smooth for $n=100,$ it is considerably more sharp and conspicuous for $n=1000.$ Overall, one can conclude 
that the branching approximation gives a useful qualitative insight into the existence of an asymptotic phase transition in the avalanche model.      
\par
In what follows we will exploit the following explicit monotone coupling of $\xn$ with a branching process.
For future convenience, we state the result in terms of a family of avalanche models rather than a single network.
At the base of the construction is a standard coupling between a binomial $B(n,p)$ and a Poisson$\bigl(-n\log(1-p)\bigr)$ random variables.
\begin{proposition}
\label{coupling1}
Let Assumption~\ref{assume7} hold. Then for every $n\geq n_0$ there exists a Markov chain $(\xc_k,\zc_k)_{k\in\zz_+}$ on $\zz_+^2$ such that the following holds true:
\item[(i)] $(\xc_k)_{k\in\zz_+}$ is distributed the same as $(X^{(n)}_k)_{k\in\zz_+}.$
\item[(ii)] $(\zc_k)_{k\in\zz_+}$ is distributed the same as $(Z^{(c)}_k)_{k\in\zz_+}.$
\item[(iii)] With probability one, $\xc_0=\zc_0$ and $\xc_k\leq \zc_k$ for all $k\in\nn.$
\end{proposition}
\begin{figure}[ht!]
\centering
\begin{tabular}{ccc}
\includegraphics[width=2.53in]{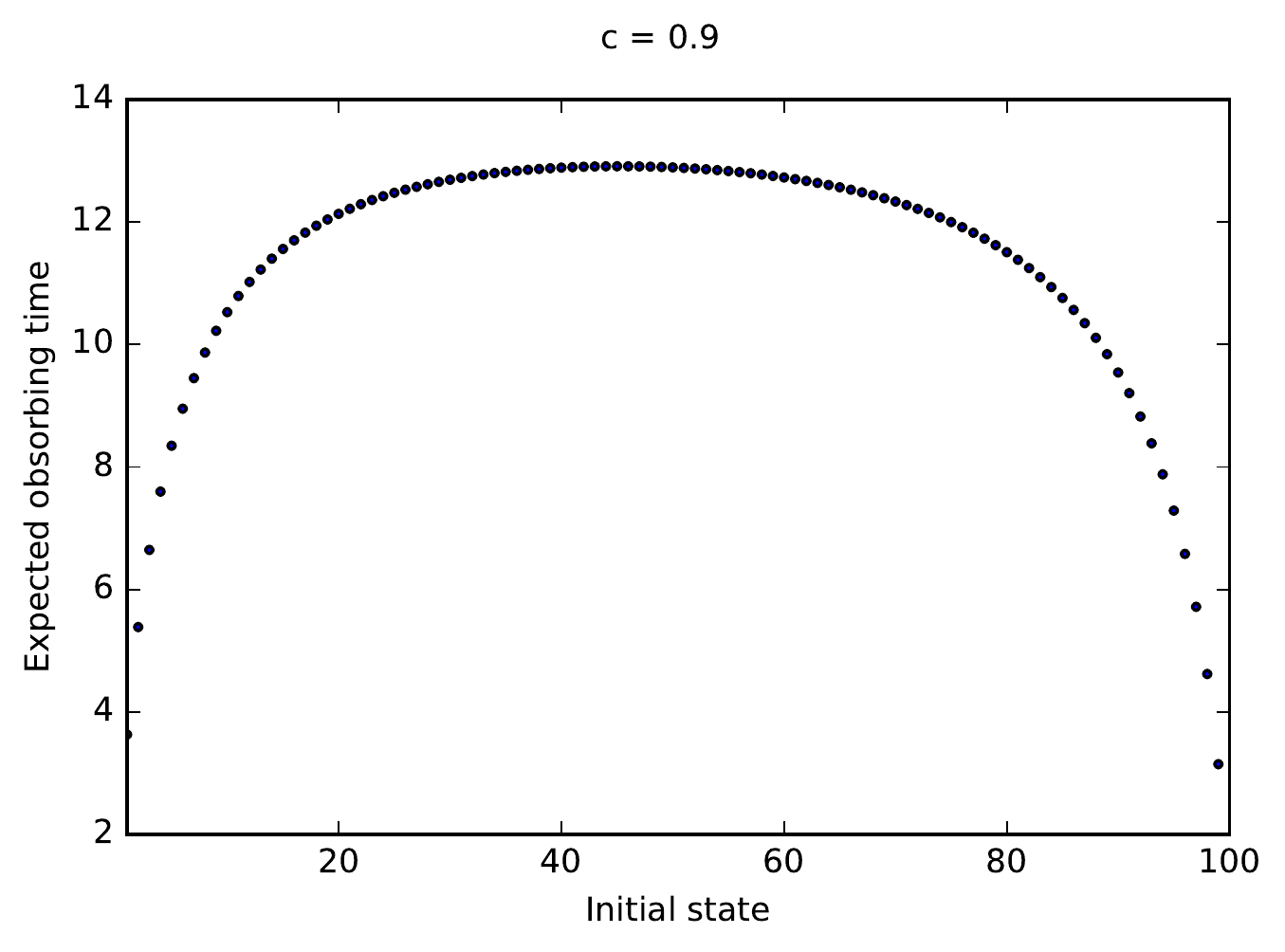}
&
$\mbox{}$\qquad $\mbox{}$
&
\includegraphics[width=2.53in]{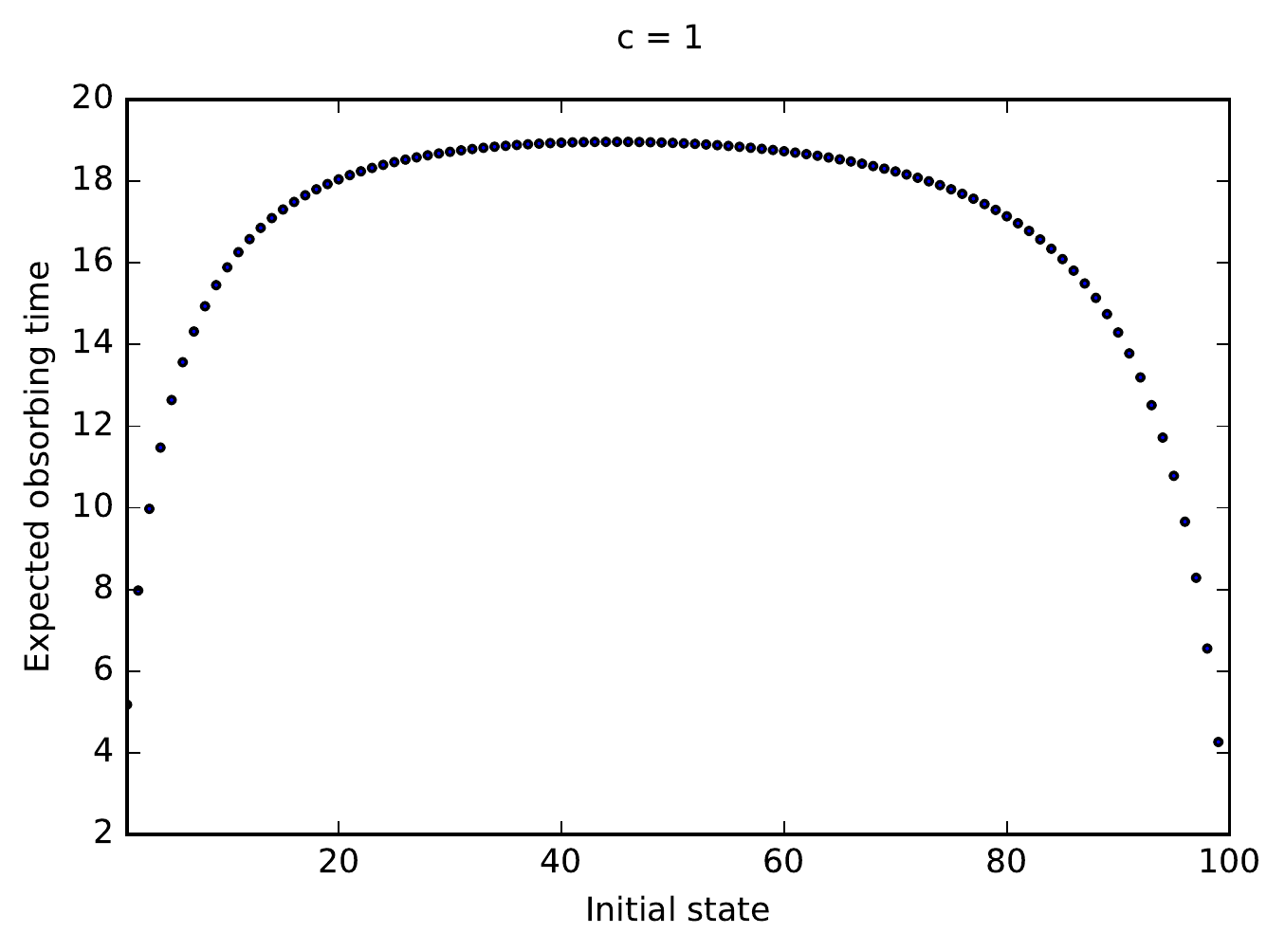}
\\
\includegraphics[width=2.53in]{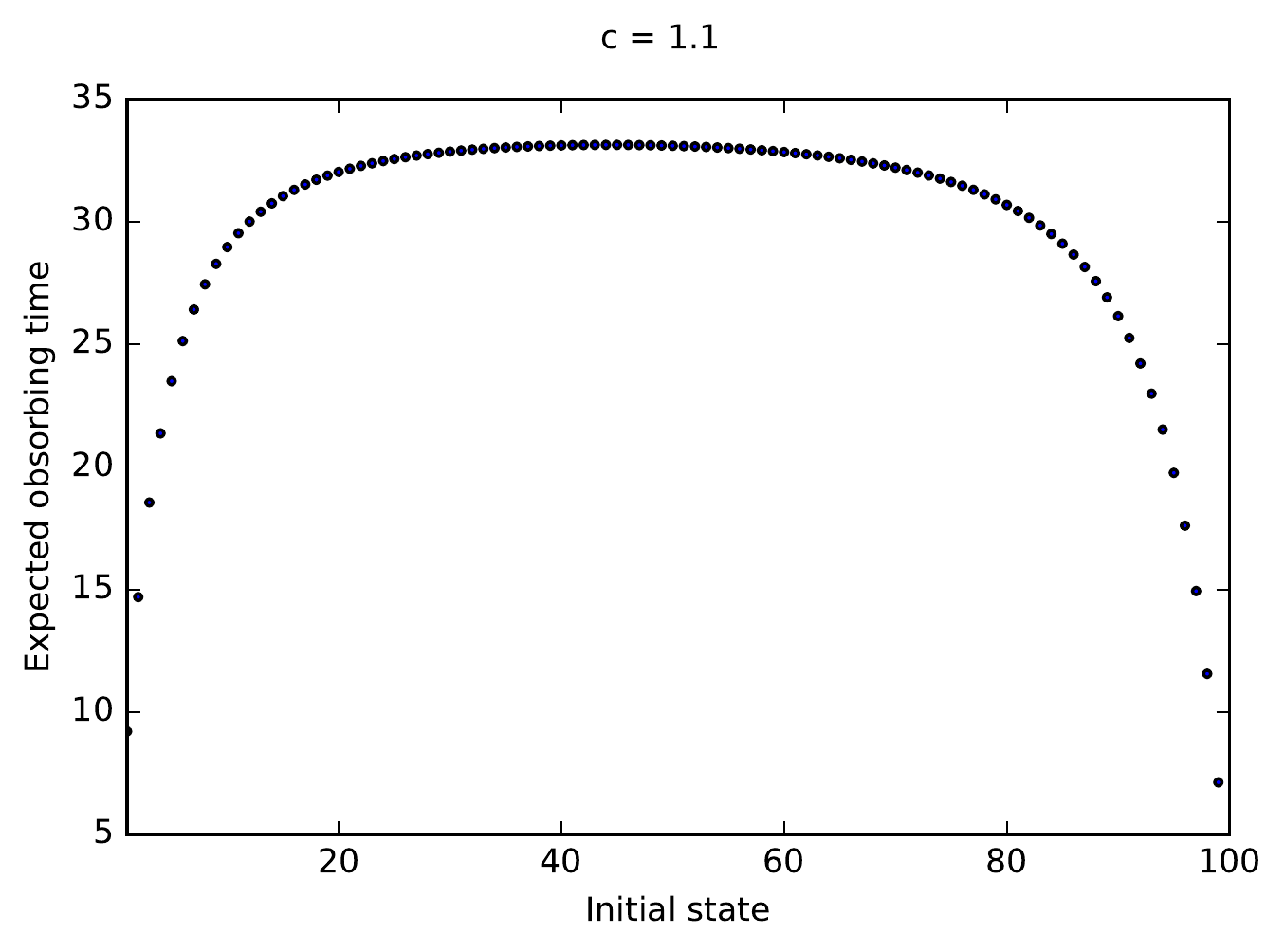}
&
$\mbox{}$\qquad $\mbox{}$
&
\includegraphics[width=2.53in]{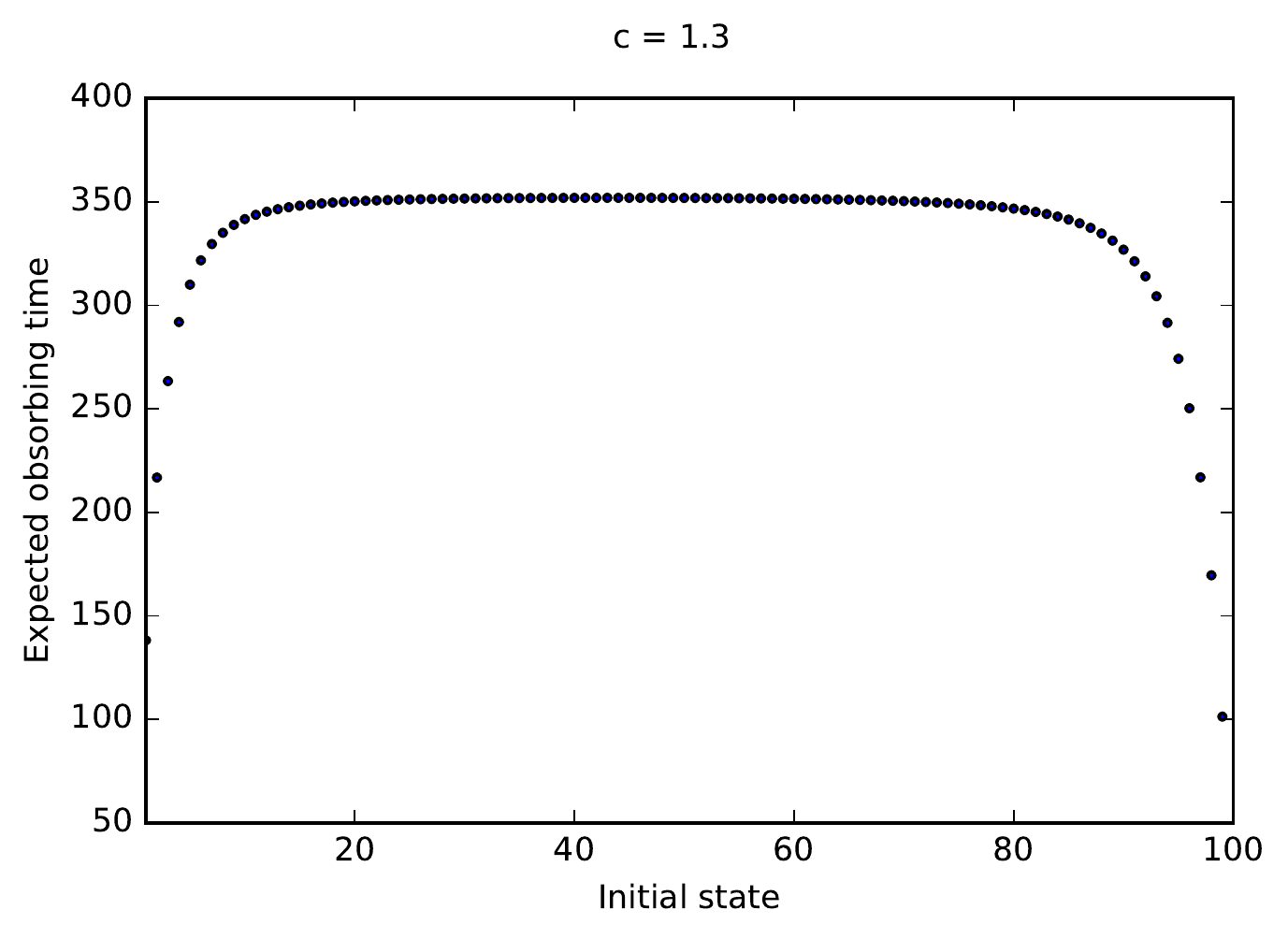}
\end{tabular}
\caption{Plot of the function $f(i_0)=E(T\,|\,X_0=i_0)$ for $n=100$ and several values of the parameter $c=np$
ranging from $c=0.9$ to $c=1.3.$ \label{f100}}
\end{figure}

\begin{figure}[H]
\centering
\begin{tabular}{ccc}
\includegraphics[width=2.53in]{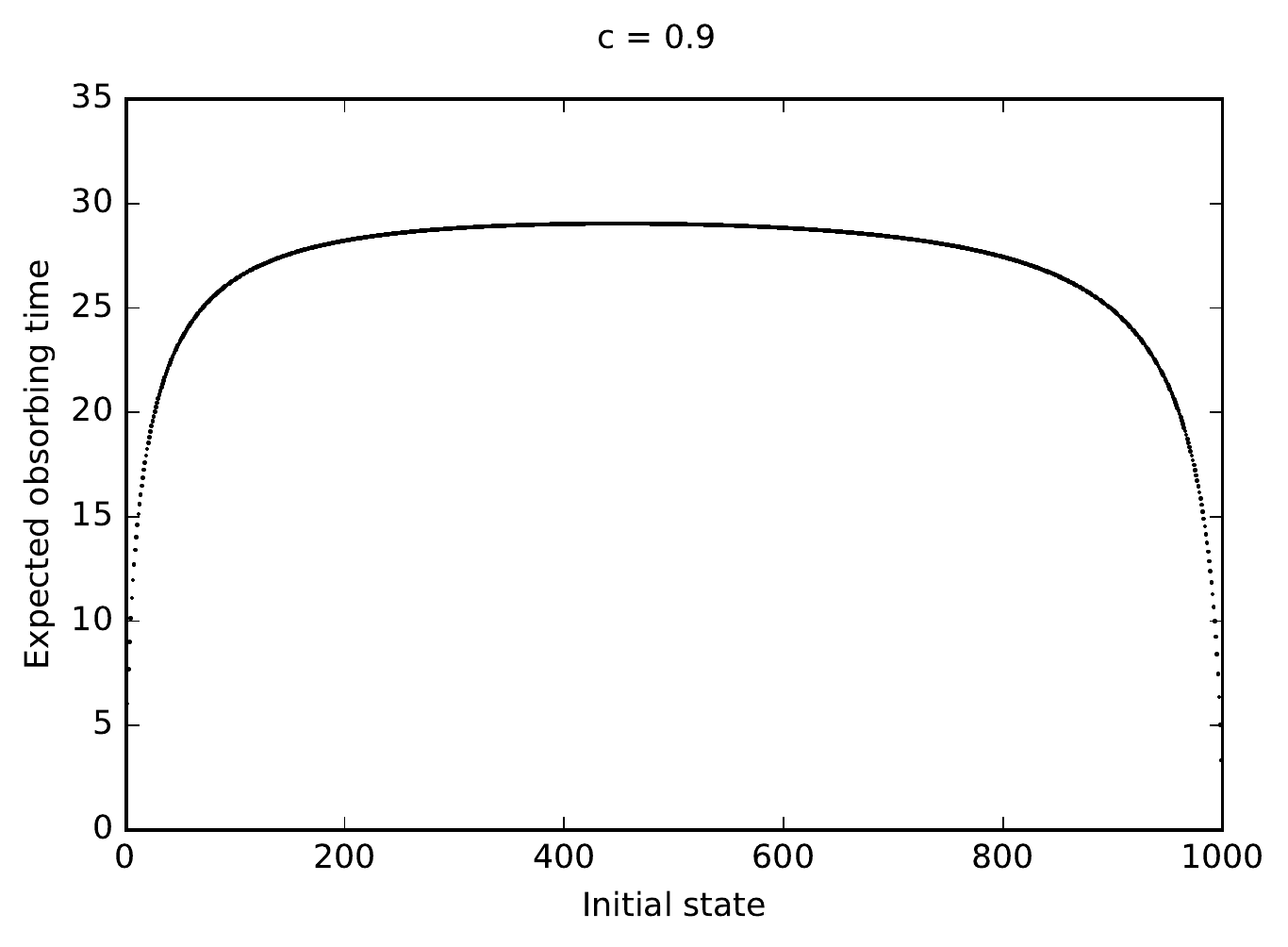}
&
$\mbox{}$\qquad $\mbox{}$
&
\includegraphics[width=2.53in]{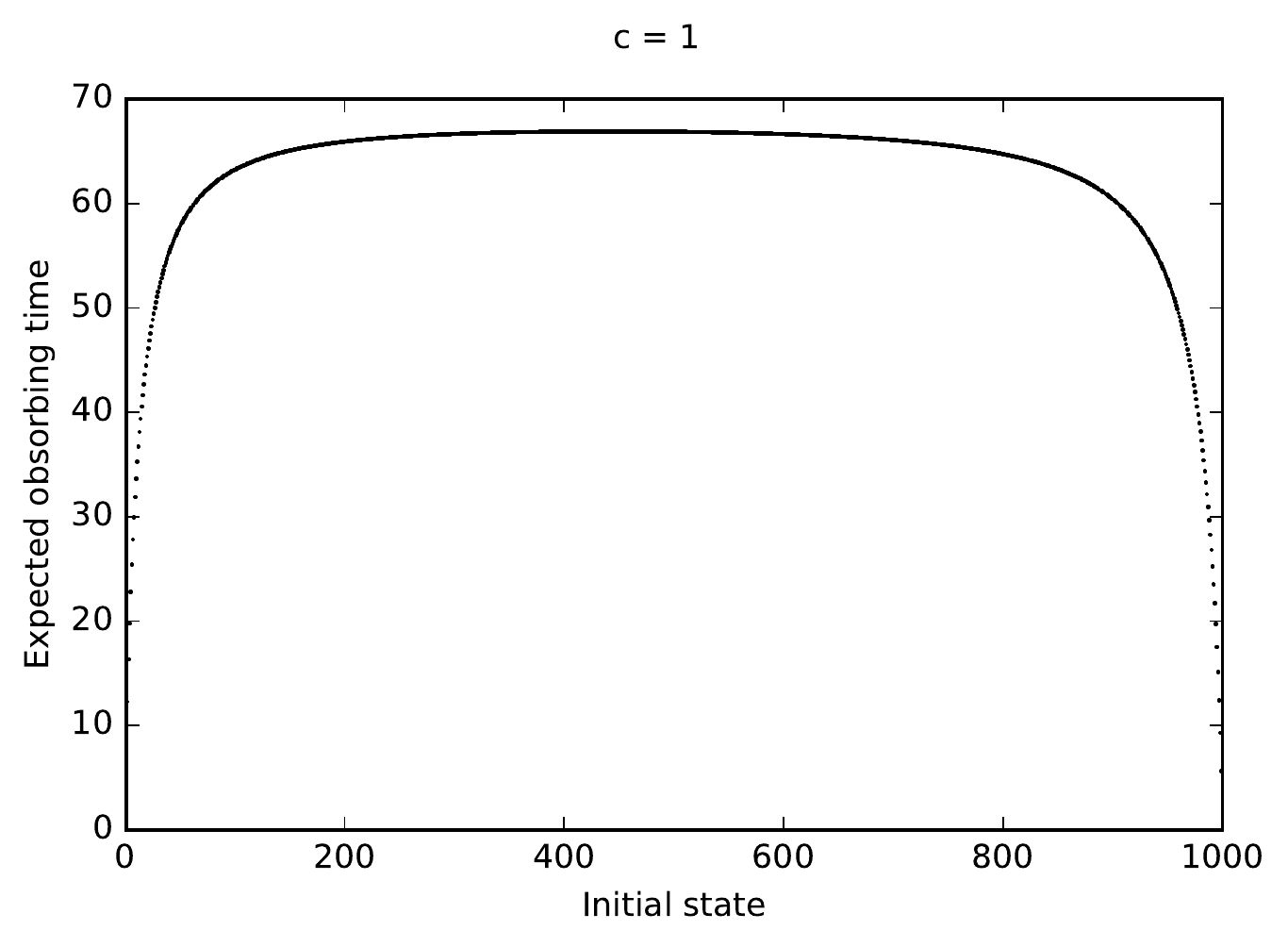}
\\
\includegraphics[width=2.53in]{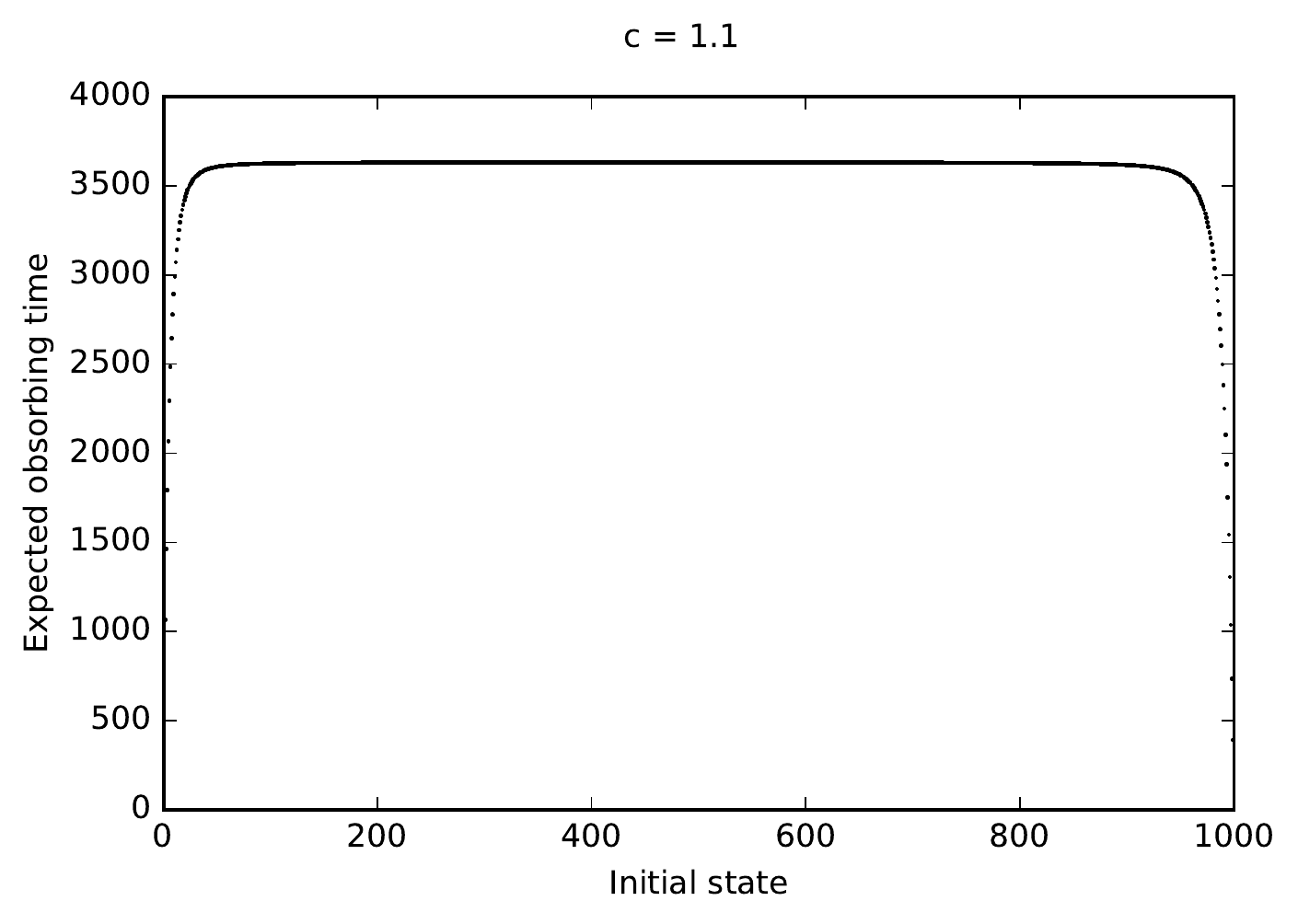}
&
$\mbox{}$\qquad $\mbox{}$
&
\includegraphics[width=2.53in]{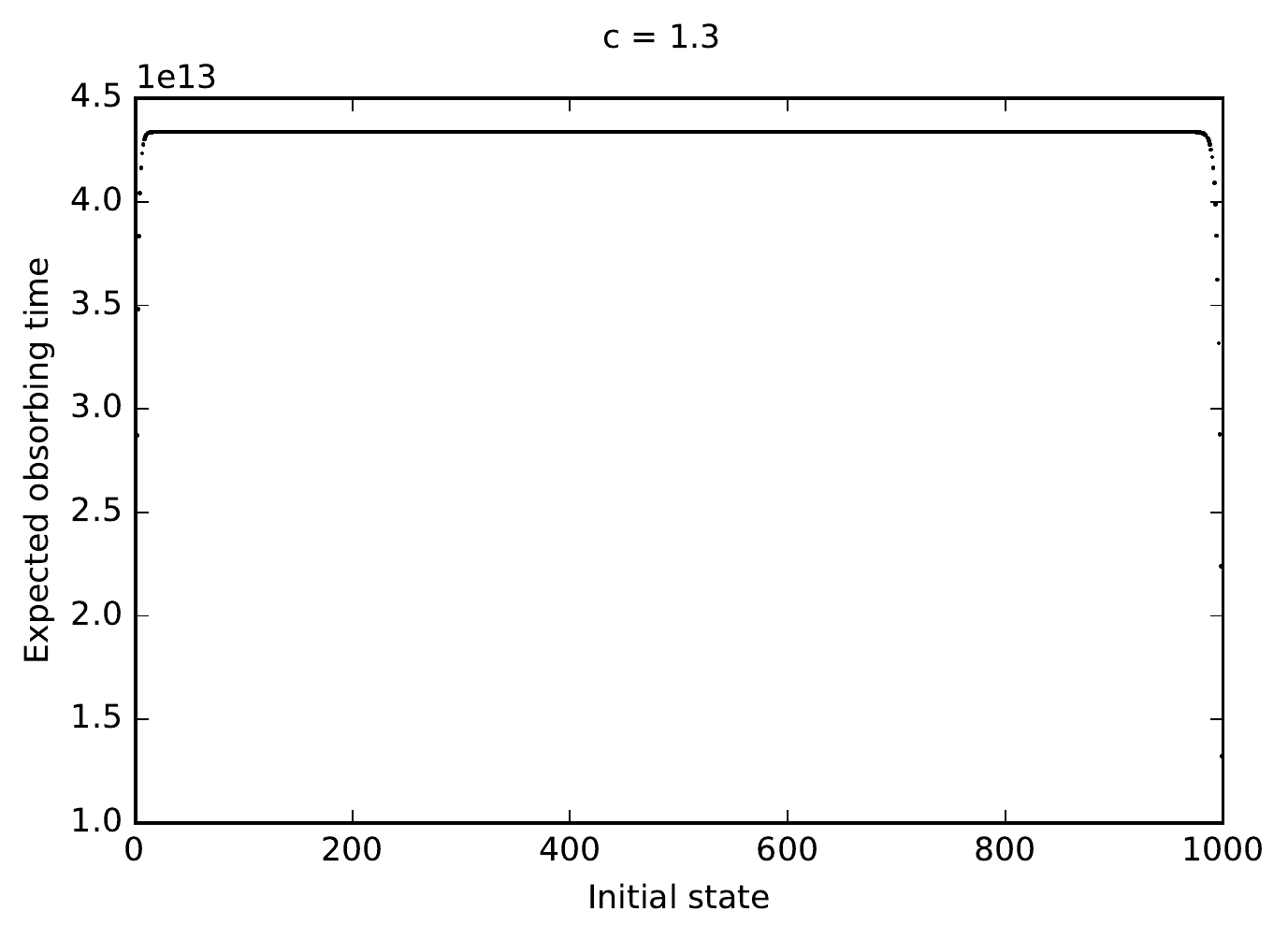}
\end{tabular}
\caption{Plot of the function $f(i_0)=E(T\,|\,X_0=i_0)$ for $n=1000$ and several values of the parameter $c=np$
ranging from $c=0.9$ to $c=`1.3.$ \label{f103}}
\end{figure}
\begin{proof}[Proof of Proposition~\ref{coupling1}]
We have
\beqn
\label{ln}
-\log q_n=-\log(1-p_n)\leq -\log \Bigl(1-\frac{c}{n}\Bigr),\qquad \forall~n\geq n_0.
\feqn
It is easy to check that for any $n\geq n_0$ and $i\in\nn,$
\beqn
\label{nol}
-\log q_n\leq  \frac{c}{n-1}\leq \frac{c}{n-i}.
\feqn
Indeed, if $f(x)=\frac{c}{x-1}+\log \Bigl(1-\frac{c}{x}\Bigr),$ then $\lim_{x\to +\infty}f(x)=0$ and for any $x>1,$
\beq
f'(x)=-\frac{c}{(x-1)^2}+\frac{c}{x^2-cx}\leq -\frac{c}{(x-1)^2}+\frac{c}{x^2-x}=\frac{c}{x-1}\Bigl[-\frac{1}{x-1}+\frac{1}{x}\Bigr]<0,
\feq
implying that $f(x)>0$ for $x>1.$  Let $\bigl\{\ycz_{n,k,j}:n\in\nn,k\in\zz_+,j\in\nn,x\in\zz_+,z\in\zz_+,x< n\bigr\}$ be a collection of independent Poisson random variables such that
\beq
P\bigl(\ycz_{n,k,j}=i)=e^{-\frac{cz}{n-x}}\frac{\bigl(\frac{cz}{n-x}\bigr)^i}{i!}, \qquad i\in\zz_+.
\feq
Further, let $U=\bigl\{\ucx_{n,k,j}:n\in\nn,k\in\zz_+,j\in\nn,x\in\zz_+,x< n\bigr\}$ be a collection of independent
Bernoulli variables which is independent of the family of Poisson variables $Y$ and such that
\beq
P\bigl(\ucx_{n,k,j}=1\bigr)=\frac{1-q_n^x}{1-e^{-\frac{cx}{n-x}}}\qquad \mbox{\rm and}\qquad P\bigl(\uc_{n,k,j}=0\bigr)=\frac{q_n^x-e^{-\frac{cx}{n-x}}}{1-e^{-\frac{cx}{n-x}}}.
\feq
Finally, set
\beq
\Bc_{ n,k,j}=\uc_{n,k,j}\one{\ycx_{k,j}>0}
\feq
and define a Markov chain of integer triples $(\xc_k,\qc_k,\zc_k)_{k\in\zz_+}$ through the initial condition $\zc_0=\qc_0=\xc_k=i_0$ and the recursion
\beqn
\label{zc}
\left\{
\begin{array}{ll}
\zc_{k+1}&=\sum_{j=1}^{n-x} \bigl(\ycx_{n,k,j}+\ycm_{n,k,j}\bigr)
\\
\mbox{}&
\\
\qc_{k+1}&=\sum_{j=1}^{n-x} \ycx_{n,k,j}
\\
\mbox{}&
\\
\xc_{k+1}&=\sum_{j=1}^{n-x}\Bc_{n,k,j}
\end{array}
\right.
\qquad
\mbox{\rm if}\quad \xc_k=x,~\zc_k=z.
\feqn
By induction, $P(\xc_k\leq \qc_k\leq \zc_k)=1$ for all $k\in\nn.$ Furthermore, by our construction $(\xc_k)_{k\in\zz_+}$ is distributed the same as $(X^{(n)}_k)_{k\in\zz_+}$ while $(\zc_k)_{k\in\zz_+}$ is distributed the same as the branching process $(Z^{(c)}_k)_{k\in\zz_+}.$
\end{proof}
Our next result concerns the total size of the avalanche, namely the total number of excited sites created by the avalanche during its entire life span. Let
\beq
S_n=\sum_{k=0}^\infty \xn_k.
\feq
Note that $P(S_n<\infty)=1$ since $\xn$ is an irreducible Markov chain with a unique absorbing state at zero. The following theorem complements
the bounds provided by Proposition~\ref{expo} for a single network with a fixed $n\in\nn.$
The theorem relates asymptotic characteristics of $\xn$ to their counterparts for the limiting branching process $Z^{(\lambda)}.$
\begin{theorem}
\label{nsumse}
Let Assumption~\ref{assume7} hold. Then
\item [(i)] \beq
\lim_{n\to\infty} E(S_n)=\left\{
\begin{array}{ll}
\frac{i_0}{1-\lambda}&\mbox{\rm if}~ \lambda \in (0,1)\\
+\infty&\mbox{\rm if}~ \lambda \geq 1.
\end{array}
\right.
\feq
\item [(ii)] If $\lambda\leq 1,$ $S_n$ converges in distribution as $n\to\infty$ to the Borel-Tanner distribution with parameters $i_0$ and $\lambda,$ that is
\beqn
\label{bt}
\lim_{n\to\infty} P(S_n=j)=\frac{i_0}{j}\frac{(\lambda j)^{j-i_o}}{(j-i_0)!}e^{-\lambda j},\qquad j\geq i_0.
\feqn
\item [(iii)]  If $\lambda> 1,$ then for any $m\geq i_0,$
\beq
\lim_{n\to \infty} P(S_n>m)=1-\alpha_\lambda+\sum_{j=m+1}^\infty \frac{i_0}{j}\frac{(\lambda j)^{j-i_o}}{(j-i_0)!}e^{-\lambda j},
\feq
where $\alpha_\lambda$ is the extinction probability of the branching process $Z^{(\lambda)},$ that is the unique in $(0,1)$ root of the fixed point
equation $\alpha_\lambda=e^{-(1-\alpha_\lambda)\lambda}.$
\item [(iv)] If $\lambda <1,$
\beq
\lim_{n\to\infty} n\Bigl[\frac{i_0}{1-\lambda}-E(S_n)\Bigr]=\frac{3i_0\lambda^2}{2(1-\lambda)^2(1+\lambda)}
+\frac{i_0^2(2\lambda-\lambda^2)}{2(1-\lambda)^2(1+\lambda)}.
\feq
\end{theorem}
For $\mu>0,$ let $S^{(\mu)}=\sum_{k=0}^\infty Z^{(\mu)}_k.$
By the Otter-Dwass theorem, the limiting distribution in \eqref{bt} is the distribution of $S^{(\lambda)}$ \cite{dwass}.
Similarly to other models using an approximation by the Poisson branching process,
the distribution tails of the total size of the underlying population (avalanche in our case) at the critical regime obey a power law.
Indeed, \eqref{bt} and Stirling's formula implies that when $\lambda=1,$ for large values of $j$ and $n,$ $P(S_n=j)$ is well-approximated by
$\frac{i_0}{j}\frac{(\lambda j)^{j-i_o}}{(j-i_0)!}e^{-\lambda j}\sim i_0\sqrt{\frac{1}{2\pi}}j^{-3/2}.$
\begin{proof}[Proof of Theorem~\ref{nsumse}]
$\mbox{}$
\\
(i) If $\lambda<1$ the result in (i) follows from Proposition~\ref{expo}. If $\lambda\geq 1,$
then a version of Fatou's lemma for weakly convergent sequences implies that for any $J\in\nn,$
\beq
\liminf_{n\to\infty} E\Bigl(\sum_{k=0}^J \xn_k\Bigr)\geq E\Bigl(\sum_{k=0}^J Z^{(\lambda)}_k\Bigr)=i_0\sum_{k=0}^J \lambda^k,
\feq
and the result follows by taking $J$ to infinity. We remark in passing that estimates similar to \eqref{sumse} show that in fact
$\lim_{n\to\infty} E\Bigl(\sum_{k=0}^J \xn_k\Bigr)=i_0\sum_{k=0}^J \lambda^k.$
\\
$\mbox{}$
\\
(ii) Let $c>0$ and $d>0$ be as in Condition~\ref{cdla}.
Assume first that $\lambda\in(0,1).$ A simple argument can be given in order to prove the result in this case.
To prove the convergence of $S_n$ to $S^{(\lambda)}$
we will consider exponential generating functions $E(e^{-\alpha S_n})$  and $E(e^{-\alpha S^{(c)}}),$ $\alpha>0,$ $c>0,$ and use the inequality
$e^{-\alpha x}-e^{-\alpha y}\leq \alpha(y-x)$ which is true for any $y>0$ and $x\in(0,y).$ It follows from \eqref{zc} that
\beq
0\leq E(e^{-\alpha S_n})-E(e^{-\alpha \Sc})\leq \alpha E(\Sc-S_n).
\feq
By Proposition~\ref{expo}, for any $n\geq n_0,$
\beq
0&\leq& \limsup_{n \to\infty}\bigl[E(e^{-\alpha S_n})-E(e^{-\alpha \Sc})\bigr]
\\
&\leq& \limsup_{n \to\infty}\bigl[E(e^{-\alpha S_n})-E(e^{-\alpha \Sc})\bigr]\leq \alpha\Bigl(\frac{i_0}{1-c}-\frac{i_0}{1-d}\Bigr),
\feq
which yields the result since the parameters $c$ and $d$ can be chosen arbitrarily close to $\lambda.$
\par
Assume now that $\lambda=1.$ Without loss of generality we can assume that $c>1.$ Let $A_c=\bigl\{\lim_{k\to\infty} Z^{(c)}_k=0\bigr\}$
be the event of extinction for the branching process  $\bigl(Z^{(c)}_k\bigr)_{k\in\zz_+}.$ It follows from Proposition~\ref{coupling1}
that for any integer $m\geq i_0,$
\beqn
\label{aam}
P(S_n>m)\leq P\bigl(S^{(c)}>m;A_c\bigr)+P\bigl(\ol A_c\bigr)=
\sum_{j>m}\frac{i_0}{j}\frac{(c j)^{j-i_o}}{(j-i_0)!}e^{-c j}+P\bigl(\ol A_c\bigr),
\feqn
where $\ol A_c$ is the complement event $A_c=\bigl\{\lim_{k\to\infty} Z^{(c)}_k=+\infty\bigr\},$ and the second identity is an instance of
the Otter-Dwass theorem for supercritical branching process, see Theorem~1 in \cite{dwass}. By letting first $n$ got to infinity and then
$c$ approach $\lambda=1,$ we obtain
\beqn
\label{lsu}
\limsup_{n\to\infty}P(S_n>m)\leq \sum_{j>m}\frac{i_0}{j}\frac{(j)^{j-i_o}}{(j-i_0)!}e^{-j}.
\feqn
On the other hand, Fatou's lemma implies that for any $J\in\nn,$
\beqn
\label{aam1}
\liminf_{n\to\infty} P(S_n>m)&\geq& \liminf_{n\to\infty} P\Bigl(\sum_{k=0}^J \xn_k>m\Bigr)\geq
P\Bigl(\sum_{k=0}^J Z^{(1)}_k>m\Bigr).
\feqn
Letting first $n$ and then $J$ go to infinity, we obtain that
\beqn
\label{m3}
\liminf_{n\to\infty}P(S_n>m)\geq P\bigl(S^{(1)}>m\bigr)=\sum_{j>m}\frac{i_0}{j}\frac{(j)^{j-i_o}}{(j-i_0)!}e^{-j}.
\feqn
Combining this estimate with \eqref{lsu} completes the proof of part (ii) for $\lambda=1.$
\\
$\mbox{}$
\\
(iii) The proof is similar to that of part (ii) for $\lambda=1.$ More precisely, \eqref{aam} and \eqref{aam1} with $Z^{(1)}$ replaced
by $Z^{(\lambda)}$ remain correct for any $m>i_0,$ $c>\lambda,$ and $J\in\nn.$ By letting first $n$ go to infinity and then $c$
approach $\lambda$ in \eqref{aam}, we obtain the following counterpart of \eqref{lsu}:
\beqn
\label{lsu1}
\limsup_{n\to\infty}P(S_n>m)\leq \sum_{j>m}\frac{i_0}{j}\frac{(j\lambda)^{j-i_o}}{(j-i_0)!}e^{-j\lambda}+1-\alpha_\lambda.
\feqn
By letting first $n$ and then $J$ go to infinity in \eqref{m3}, we obtain from the Otter-Dwass theorem that
\beq
\liminf_{n\to\infty}P(S_n>m)\geq P\bigl(S^{(\lambda)}>m\bigr)=\sum_{j>m}\frac{i_0}{j}\frac{(j\lambda)^{j-i_o}}{(j-i_0)!}e^{-j\lambda}+1-\alpha_\lambda.
\feq
Combining this estimate with \eqref{lsu1} completes the proof of part (iii) of the theorem.
\\
$\mbox{}$
\\
(iv) Using the Lagrange form of the second order remainder in Taylor's series for the function $f(p)=(1-p)^i$ around zero, we obtain that
for all $n\in\nn$ and $i\in\nn,$
\beq
q_n^i=(1-p_n)^i=1-ip_n+\frac{i(i-1)}{2}p_n^2(1-\beta_{n,i})^{i-2}
\feq
for some $\beta_{n,i}\in (0,p_n).$ Therefore,
\beqn
\nonumber
E(\xn_{k+1})&=& E\bigl[(n-\xn_k)\bigl(1-q_n^{\xn_k}\bigr)\bigr]
\\
\nonumber
&=&
np_nE(\xn_k) +\frac{np_n^2}{2}E\bigl[\xn_k(\xn_k-1)\bigr]
\\
\label{sume1}
&&
\qquad
-p_nE[(\xn_k)^2]+\frac{p_n^2}{2}E\bigl[(\xn_k)^2(\xn_k-1)(1-\beta_{n,\xn_k})^{\xn_k-2}\bigr].
\feqn
It follows from the coupling construction given by Proposition~\ref{coupling1} that
\beq
\sum_{k=0}^\infty E\bigl[(\xn_k)^3\bigr]<\sum_{k=0}^\infty E\bigl[(Z^{(c)}_k)^3\bigr]<\infty
\feq
for any $c\in (\lambda,1)$ and $n\geq n_0.$ Hence, summing up both the sides of \eqref{sume1} from $k=1$ to infinity, we obtain
\beq
E(S_n)&=&\frac{i_0}{1-np_n}+\frac{np_n^2}{2(1-np_n)}E\Bigl[\sum_{k=0}^\infty \xn_k(\xn_k-1)\Bigr]
\\
&&
\qquad
-\frac{p_n}{1-np_n}E\Bigl[\sum_{k=0}^\infty (\xn_k)^2\Bigr]+o(1/n).
\feq
Therefore, by the dominated convergence theorem (here we use again Proposition~\ref{coupling1} which shows that $\xn_k$
is stochastically dominated by $Z^{(c)}_k$),
\beqn
\nonumber
\lim_{n\to\infty}n\Bigl[\frac{i_0}{1-c}-E(S_n)\Bigr]&=&-\frac{\lambda^2}{2(1-\lambda)}E\Bigl[\sum_{k=0}^\infty (Z^{(\lambda)}_k)^2-\frac{i_0}{1-\lambda}\Bigr]+\frac{\lambda}{1-\lambda}E\Bigl[\sum_{k=0}^\infty (Z^{(\lambda)}_k)^2\Bigr]
\\
\label{moments}
&=&
\frac{i_0\lambda^2}{2(1-\lambda)^2}+\frac{2\lambda-\lambda^2}{2(1-\lambda)}E\Bigl[\sum_{k=0}^\infty (Z^{(\lambda)}_k)^2\Bigr].
\feqn
Furthermore,
\beq
E\Bigl[\sum_{k=0}^\infty (Z^{(\lambda)}_k)^2\Bigr]=i_0^2+E\Bigl[\sum_{k=1}^\infty \bigl[\lambda Z^{(\lambda)}_{k-1}
+ \lambda^2 (Z^{(\lambda)}_{k-1})^2\bigr]\Bigr],
\feq
which implies that (the sum is finite, for instance, because it is dominated by a finite second moment of the Borel-Tanner distribution of $S^{(\lambda)}$)
\beq
E\Bigl[\sum_{k=0}^\infty (Z^{(\lambda)}_k)^2\Bigr]=\frac{\lambda i_0}{(1-\lambda)(1-\lambda^2)}+\frac{i_0^2}{1-\lambda^2}.
\feq
Substituting this identity into \eqref{moments} yields the result in part (iv) of the theorem.
\end{proof}
\section{Spread to a non-zero fraction of the network}
\label{fixp}
In this section we are concerned with the question whether an avalanche initiated by just a few excited nodes has a substantial potential to spread to
a large fraction of the network. The results for the supercritical regime are stated
in Theorems~\ref{main1} and~\ref{newt}, whereas the critical and subcritical regimes are addressed in Theorem~\ref{lem3a1}.
\par
First we will consider a single network with given parameters $n$ and $p.$  For an arbitrary real number $J>0$ we define
\beqn
\label{pni}
h_J(i)=P\Bigl(\max_{k\in\zz_+} X_k \geq J\Bigl|X_0=i\Bigr)=P\bigl(X_{T_J} \geq J|X_0=i\bigr),
\feqn
where
\beqn
\label{tiveps}
T_J=\min\bigl\{k\in\nn:X_k=0~\mbox{or}~X_k \geq J\bigr\}.
\feqn
We begin with the supercrtical regime, namely the case when $np>1.$ Consequently, without loss of generality we can assume that $d>1$ in Condition~\ref{assume5}.
Let $\veps_d>0$ be a positive constant such that
\beqn
\label{epss}
\frac{1-e^{-d\veps}}{\veps}(1-\veps)>1 \qquad \forall~\veps\in (0,\veps_d].
\feqn
Note that such $\veps_d$ exists because $\lim_{\veps\to 0} \frac{1-e^{-d\veps}}{\veps}(1-\veps)=d>1.$
Further, for $\mu\in (0,\infty)\backslash\{1\},$ let $\alpha_\mu\neq 1$ denote the unique in $(0,\infty)\backslash\{1\}$ solution of the fixed point equation
\beqn
\label{alpha}
\alpha_\mu=e^{-(1-\alpha_\mu)\mu},\qquad \alpha_\mu\neq 1.
\feqn
We remark that $\alpha_\mu<\mu\alpha_\mu<1$ if $\mu>1,$ and $\alpha_\mu>\mu\alpha_\mu>1$ if $\mu\in (0,1).$
This is true because \eqref{alpha} is equivalent to
$\mu\alpha_\mu e^{-\mu\alpha_\mu}=\mu e^{-\mu},$ and the function $f(x)=xe^{-x}$ has a unique local maximum at $x=1.$
\par
Observe that the right-hand side of \eqref{alpha} is $E\Bigl(\alpha_\mu^{Z^{(\mu)}_{k+1}}\Bigl|Z^{(\mu)}_k=1\Bigl),$ and hence
$\Bigl(\alpha_\mu^{Z^{(\mu)}_k}\Bigr)_{k\in\zz_+}$ is a martingale with respect to its natural filtration.
If $\mu>1,$ then $\alpha_\mu$ is the extinction probability of the supercritical Poisson branching process $Z^{(\mu)}.$ Furthermore, if $\mu<1$
then $1/\alpha_\mu$ is the extinction probability of the dual supercritical process $Z^{(\mu\alpha_\mu)}.$
\par
The main technical result of this section is the following proposition.
\begin{proposition}
\label{lem2}
Suppose that Condition~\ref{assume5} is satisfied with $d>1.$ Let $\veps_d$ be a constant that satisfies condition \eqref{epss}. Then
\begin{itemize}
\item[(a)]
There is a constant $\rho=\rho(\veps_d)\in (0,1)$ such that $E\bigl(\rho^{X_{k+1}}|X_k=i\bigr)\leq \rho^i$ for all $i\in [1,n\veps_d)$ and $k\in \zz_+.$ Furthermore, $\rho(\veps_d)$ can be chosen in such a way that
\beqn
\label{rhov}
\lim_{\veps_d\to 0} \rho(\veps_d)=\alpha_d.
\feqn
\item[(b)] $E\bigl(\alpha_c^{X_{k+1}}|X_k=i\bigr)\geq \alpha_c^i$ for all integers $i\in [1,n)$ and $k\in \zz_+.$
\end{itemize}
\end{proposition}
\begin{proof}
$\mbox{}$
\\
(a) First, we choose a real constant $\gamma>0$ in such a way that
\beq
q^i\leq \Bigl(1-\frac{d}{n}\Bigr)^i\leq 1-\frac{\gamma i}{n}
\feq
for any integer $i\in [1,n\veps_d).$  Since
\beq
q^i\leq \Bigl(1-\frac{d}{n}\Bigr)^i=\Bigl[\Bigl(1-\frac{d}{n}\Bigr)^n\Bigr]^{\frac{i}{n}}\leq e^{-\frac{di}{n}},
\feq
it suffices to find $\gamma >0$ such that
\beq
e^{-dx}\leq 1-\gamma x
\feq
for any $x\in [0,\veps_d).$ To this end, for a fixed $\gamma>0$ let $f_\gamma(x)=1-\gamma x-e^{-dx}.$
Then $f_\gamma(0)=0$ and
\beq
f_\gamma'(x)=-\gamma+de^{-dx} \qquad \mbox{\rm and}\qquad f_\gamma''(x)=-d^2e^{-dx}<0.
\feq
Thus $f_\gamma(x)>0$ for any $x\in (0,\veps_d)$ provided that
\beq
f_\gamma'(0)=-\gamma+d>0\qquad \mbox{\rm and}\qquad f_\gamma(\veps_d)=1-\gamma\veps_d -e^{-d\veps_d}\geq 0.
\feq
Since $\frac{1-e^{-d\veps_d}}{\veps_d}<d,$ we can put
\beqn
\label{gamma}
\gamma=\frac{1-e^{-d\veps_d}}{\veps_d}.
\feqn
Then for any $\rho\in (0,1)$ and integer $i\in [1,n\veps_d),$
\beqn
\nonumber
E\bigl(\rho^{X_{k+1}}\bigr|X_k=i\bigr)&=& \bigl(\rho (1-q^i)+1\cdot q^i\bigr)^{n-i}=\bigl(\rho +(1-\rho)\cdot q^i\bigr)^{n-i}
\\
\nonumber
&\leq&
\Bigl(\rho +(1-\rho)\cdot \Bigl(1-\frac{\gamma i}{n}\Bigr)\Bigr)^{n-i}=\Bigl(1 -(1-\rho)\frac{\gamma i}{n}\Bigr)^{n-i}
\\
\label{rhog}
&\leq& e^{-(1-\rho)\gamma(1-\veps_d) i}.
\feqn
Thus, we can set $\rho$ to be the unique solution of the fixed point equation
\beqn
\label{rho}
\rho= e^{-(1-\rho)\gamma (1-\veps_d)}.
\feqn
Note that $\rho=\alpha_\mu$ with $\mu=\gamma (1-\veps_d)=\frac{1-e^{-d\veps_d}}{\veps_d} (1-\veps_d).$ The limit result in \eqref{rhov}
follows immediately from \eqref{rho} and \eqref{gamma}.
\\
$\mbox{}$
\\
(b) By Proposition~\ref{coupling1} the process $Z^{(c)}$ stochastically dominates $X.$ Therefore, taking in account that $\alpha_c<1.$ we obtain
\beq
E\bigl(\alpha_c^{X_{k+1}}|X_k=i\bigr)\geq E\bigl(\alpha_c^{Z^{(c)}_{k+1}}|Z^{(c)}_k=i\bigr)=\alpha_c^i.
\feq
The proof of the proposition is complete.
\end{proof}
Doob's optional stopping theorem implies that for any $\veps\in (0,\veps_d)$ and $i\in [1,n\veps),$
\beq
\rho^i\geq E\bigl(\rho^{X_{T_{n\veps}}}|X_0=i\bigr)\geq h_{n\veps}(i)\rho^n+\bigl(1-h_{n\veps}(i)\bigr)
\feq
and
\beq
\alpha_c^i\leq E\bigl(\alpha_c^{X_{T_{n\veps}}}|X_0=i\bigr)\leq h_{n\veps} (i)\alpha_c^{n\veps}+\bigl(1-h_{n\veps} (i)\bigr).
\feq
This yields the following result:
\begin{theorem}
\label{main1}
Suppose that Condition~\ref{assume5} is satisfied with $d>1.$  Then
\beq
\frac{1-\rho^i}{1-\rho^n}\leq h_{n\veps}(i)\leq \frac{1-\alpha_c^i}{1-\alpha_c^{n \veps}},
\feq
for all $\veps\in (0,\veps_d)$ and $i\in [1, n \veps).$
\end{theorem}
We next consider the asymptotic behavior of the avalanche model when the network size approaches infinity. Similarly to \eqref{pni} and \eqref{tiveps},
for any $J>0$ and $n\in\nn$ we define
\beqn
\label{pnin}
h^{(n)}_J(i)=P\Bigl(\max_{k\in\zz_+} X^{(n)}_k \geq J\Bigl|X^{(n)}_0=i\Bigr)=P\bigl(X^{(n)}_{T^{(n)}_J} \geq J\bigl|X^{(n)}_0=i\bigr),
\feqn
where
\beqn
\label{tiveps1}
T^{(n)}_J=\min\bigl\{k\in\nn:X^{(n)}_k=0~\mbox{or}~X^{(n)}_k \geq J \bigr\}.
\feqn
Suppose that Assumption~\ref{assume7} (and consequently Condition~\ref{cdla}) hold with $\lambda>1.$ Theorem~\ref{main1} then implies that for any constants $i\in \nn,$  $\veps\in (0,\veps_d),$ and a function $\psi:\nn\to\nn$ such that
\beqn
\label{psiz}
\lim_{n\to\infty} \psi(n)=+\infty\qquad \mbox{\rm and}\qquad \limsup_{n\to\infty} \frac{\psi(n)}{n}=0,
\feqn
we have
\beq
1-\rho^i \leq \liminf_{n\to\infty} h^{(n)}_{\psi(n)}(i)\leq \limsup_{n\to\infty} h^{(n)}_{\psi(n)}(i)\leq 1-\alpha_c^i,
\feq
where $\rho=\rho(\veps),$ as defined in the statement of part (a) of Proposition~\ref{lem2}. Because of the second condition in \eqref{psiz} we can chose the constant $\veps>0$ to be as small as we wish. Therefore, by virtue of \eqref{rhov},
\beq
1-\alpha_d^i \leq \liminf_{n\to\infty} h^{(n)}_{\psi(n)}(i)\leq \limsup_{n\to\infty} h^{(n)}_{\psi(n)}(i)\leq 1-\alpha_c^i.
\feq
Since the constants $c$ and $d$ can be chosen arbitrarily close to $\lambda,$ we arrive to
\beqn
\label{hpsi}
\lim_{n\to\infty} h^{(n)}_{\psi(n)}(i)=1-\alpha_\lambda^i.
\feqn
It turns out that condition \eqref{psiz} can be relaxed as follows:
\begin{theorem}
\label{newt}
Suppose that Assumption~\ref{assume7} is satisfied with $\lambda>1.$  Then \eqref{hpsi} holds for any integer $i\in \nn$ and a function $\psi:\nn\to \nn$ such that $\psi(n)< n$ and $\lim_{n\to\infty} \psi(n)=+\infty.$
\end{theorem}
A similar result for a frequency-dependent Wright-Fisher model has been obtained in \cite{jmb}.
The proof of the theorem is very similar to that of Theorem~3.8 in \cite{jmb}, and therefore is omitted.
We remark that the proof requires a uniform in $n$ upper bound estimate on $P\bigl(X^{(n)}_{k+1}=0\bigl|X^{(n)}_k=m\bigr)$ for given $m\in\nn.$
One can use, for instance, the following bound:
\beq
P\bigl(X^{(n)}_{k+1}=0\bigl|X^{(n)}_k=m\bigr)=q^{m(n-m)}\geq \Bigl(1-\frac{c}{n}\Bigr)^{m(n-m)}\geq e^{-2cm}
\feq
for all $n$ large enough.
\par
We now turn to the study of the maximal number of excited sites in the subcritical and critical regimes.
\begin{theorem}
\label{lem3a1}
Let Assumption~\ref{assume7} hold.
\item [(a)] If $\lambda<1,$ then for any integers $i\in\nn $ and $m>i,$
\beqn
\label{m}
\limsup_{n\to\infty} h^{(n)}_m(i)\leq \frac{\alpha_\lambda^i-1}{\alpha_\lambda^m-1}.
\feqn
\item [(b)] If $\lambda=1,$ then for any integers $i\in\nn $ and $m>i,$
\beqn
\label{m1}
\limsup_{n\to\infty} h^{(n)}_m(i)\leq \frac{i}{m}.
\feqn
\end{theorem}
\begin{remark}
\label{maxrem}
Let $M^{(\lambda)}=\max_{k\in\zz_+} Z^{(\lambda)}_k$ and $M_n=\max_{k\in\zz_+} X^{(n)}_k.$
The estimates on the right-hand side of \eqref{m} and \eqref{m1} are classical upper bounds for
$P\bigl(M^{(\lambda)}\geq m\bigl| Z^{(\lambda)}_0=i\bigr)$ \cite{lind}.
Note that the estimates are not trivial in the sense that, in general, $X^{(n)}$ is not dominated by the limiting branching process $Z^{(\lambda)}$
because it is possible that $np_n>\lambda.$ However, since both $M^{(\lambda)}$ and $M_n$ are a-priori finite with probability one,
\eqref{kernelc} implies that $M_n$ converges to $M^{(\lambda)}$ in distribution as $n\to\infty.$
Hence one can expect that the bounds are meaningful for the avalanche model when $n$ is large.
The bound in \eqref{m1} is known to be asymptotically accurate as $m\to\infty,$ namely
$\lim_{m\to\infty} P\bigl(M^{(1)}\geq m\bigl| Z^{(1)}_0=i\bigr)=\frac{i}{m}$ \cite{lind}.
For the subcritical process, Theorem~\^2 in \cite{nerd} suggests that the correct order of
$P\bigl(M^{(\lambda)}\geq m\bigl| Z^{(\lambda)}_0=i\bigr)$ as $m\to\infty$ is $m^{-1}\alpha_\lambda^{-m},$ up to a constant that depends on $i.$
\end{remark}
Before we prove Theorem~\ref{lem3a1} we state a direct consequence for our model of some well-known results for branching processes
in the critical and subcritical regimes. The first part is an implication of a result in \cite{lind} mentioned in Remark~\ref{maxrem},
the second part can be derived from a result of \cite{kbam}, and the third one follows from Theorem~\^2 in \cite{nerd}, all three are based
on a comparison to branching process invoking Proposition~\ref{coupling1}.
We continue to use the notation for maxima of a branching process introduced in the above remark.
\begin{proposition}
\label{maxc}
There exists a sequence of positive constants $\{\veps_m>0:m\in\nn\}$
such that $\lim_{m\to\infty} \veps_m=0$ and the following holds true:
\item [(a)]  If $np=1$ in \eqref{ker}, then
\begin{itemize}
\item [(i)] $P\bigl(\max_{k\in \zz_+} X_k> m)\leq \frac{i_0}{m}\bigl(1+\veps_m\bigr)$ for all integer $m\in (i_0,n).$
\item [(ii)] $E\bigl(\max_{0\leq j\leq k} X_j)\leq \log k\bigl(1+\veps_k\bigr)$ for all $k\in\nn.$
\end{itemize}
\item [(b)] If $c:=np<1$ in \eqref{ker}, then there exists a constant $B=B(c,i_0)>0$ that depends
on $c$ and $i_0$ only (but not on $n$ and $p$) such that
$P\bigl(\max_{k\in \zz_+} X_k> m)\leq \frac{B(c,i_0)}{m\alpha_c^m}\bigl(1+\veps_m\bigr)$ for all integer $m\in (i_0,n).$
\end{proposition}
We remark that we chose the sequence $\veps_m$ in the statement of the proposition to be the same in all three cases exclusively for
a notational convenience.
\par
We now return to Theorem~\ref{lem3a1}.
\begin{proof}[Proof of Theorem~\ref{lem3a1}]
$\mbox{}$
\\
(a) Let $n_0\in\nn$ and $d\in (0,\lambda)$ satisfy Condition~\ref{cdla}. Fix any $\veps>0$ and, similarly to \eqref{gamma}, let
\beqn
\label{gamma1}
\gamma=\gamma(\veps)=\frac{1-e^{-d\veps}}{\veps}.
\feqn
Notice that $\gamma<d<1.$ Similarly, to \eqref{rhog}, for any $n\geq n_0,$ $\rho>1$ and integer $i\in [1,n\veps),$
\beq
E\bigl(\rho^{X^{(n)}_{k+1}}\bigr|X^{(n)}_k=i\bigr)\leq
\Bigl(1 +(\rho-1)\frac{\gamma i}{n}\Bigr)^{n-i}\leq e^{(\rho-1)\gamma i}.
\feq
Recall \eqref{alpha} and set $\rho=\alpha_\gamma.$
Thus $E\bigl(\rho^{X^{(n)}_{k+1}}\bigr|X^{(n)}_k=i\bigr)\leq \rho^i$ for any $n\geq n_0$ and $\in [1,n\veps).$ Doob's optional theorem implies that
\beq
\rho^i\geq E\Bigl(\rho^{X^{(n)}_{T^{(n)}_m}}\Bigl|X^{(n)}_0=i\Bigr)\geq h^{(n)}_m(i)\rho^m+\bigl(1-h^{(n)}_m(i)\bigr).
\feq
Therefore,
\beq
h^{(n)}_m(i)\leq \frac{\rho^i-1}{\rho^m-1}=\frac{\alpha_\gamma^i-1}{\alpha_\gamma^m-1}\qquad \forall~n\geq n_0.
\feq
By taking $\veps$ to zero in \eqref{gamma1} we obtain
\beqn
\label{hc}
\limsup_{n\to\infty} h^{(n)}_m(i)\leq \frac{\alpha_d^i-1}{\alpha_d^m-1}.
\feqn
Since in this argument $d$ is an arbitrary number in $(0,\lambda),$ the result of part (a) follows by taking $d$ to $\lambda$ in the above inequality.
\\
$\mbox{}$
\\
(b) Observe that \eqref{hc} is still true for $\lambda=1$ and any $d\in (0,1).$ Furthermore, $\lim_{d\to 1} \alpha_d=1,$ and hence, by the L'Hospital rule,
\beq
\limsup_{n\to\infty} h^{(n)}_m(i)\leq \lim_{\alpha\to 1}\frac{\alpha^i-1}{\alpha^m-1}=\frac{i}{m}.
\feq
The proof of the proposition is complete.
\end{proof}
\section{Duration of the avalanche}
\label{dtime}
The goal of this section is to evaluate distribution tails of the avalanche's duration. The expected value of the duration
is discussed in some detail elsewhere \cite{longini, duration} using a mixture of numerical and computational approaches.
The main result here is Theorem~\ref{main7} which provides estimates for a single network.
Some consequences for a network ensemble satisfying Assumption~\ref{assume7} are drawn in Corollaries~\ref{th7} and~\ref{tc7} at the end of the section.
The basic idea of the proofs is to compare the avalanche model to a branching process during a time frame which on one hand is large enough for
an asymptotic pattern to emerge and, on the other hand, is sufficiently small so that with an asymptotically overwhelming probability,
the paths of the avalanche Markov chain $X$ and a coupled branching process wouldn't diverge until its end.
The proofs rely on asymptotically tight estimates for the branching process given in \cite{agresti74}.
\par
For $c\in (0,1)$ let
\beqn
\label{s74}
s(s)=\frac{2-c}{c} \qquad \mbox{and} \qquad r(c)=\frac{ce^{-c}}{e^{-c}-(1-c)}.
\feqn
Recall $T$ from \eqref{dusize} and $\alpha_c$ from \eqref{alpha}.
\begin{theorem}
\label{main7}
Consider an avalanche process $(X_k)_{k\in\zz_+}$ with $X_0=i_0.$  Let $c=np.$
\item [(i)] If $c>1,$ then there exist constants $\theta >0$ and $K >0$ such that for any pair of constants $x>0$ and $m\in\nn$
which satisfies the condition $xc^m<n,$
\beq
\Bigl(\frac{\alpha_c s_1(1-c^m\alpha_c^m)}{s_1-c^m\alpha_c^m}\Bigr)^{i_0}\leq
P(T\leq m)\leq \Bigl(\frac{\alpha_c r_1 (1-c^m\alpha_c^m)}{r_1-c^m\alpha_c^m}\Bigr)^{i_0}+\frac{3c^{\frac{3(m+1)}{2}}mx^{3/2}}{2n}+K^{i_0}e^{-\theta x},
\feq
where $s_1=s(\alpha_cc)$ and $r_1=r(\alpha_cc).$
\item [(ii)] If $c<1,$ then for any pair of constants $m,J\in\nn$  such that $m<n$ and $i_0<J<n,$
\beq
\Bigl(\frac{s_2(1-c^m)}{s_2-c^m}\Bigr)^{i_0}\leq
P(T\leq m)\leq \Bigl(\frac{r_2(1-c^m)}{r_2-c^m}\Bigr)^{i_0}+\frac{3cmJ^2}{2n}+\frac{\alpha_c^{i_0}-1}{\alpha_c^J-1},
\feq
where $s_2=s(c)$ and $r_2=r(c).$
\item [(iii)] If $c=1,$ then for any $m\in\nn$  such that $m<n,$
\beq
\Bigl(\frac{m}{m+2}\Bigr)^{i_0}\leq
P(T\leq m)\leq \Bigl(\frac{m}{m+e-1}\Bigr)^{i_0}+\frac{3}{2}\Bigl(\frac{3mi_0^2}{n}\Bigr)^{1/3}.
\feq
\end{theorem}
The above bounds for the distribution function of $T$ are originated in their counterparts for the limiting branching process, see Lemma~\ref{tbounds} below.
The latter estimates are borrowed from \cite{agresti74}. We remark that in the supercritical regime $c>1,$ a similar result for a
frequency-dependent Wright-Fisher model has been proved in \cite{jmb} (see Theorem~3.9 there). By taking $n$ to infinity in the conclusions of the theorem,
one can obtain tight asymptotic bounds for the avalanche model under Assumption~\ref{assume7}, see Corollaries~\ref{th7} and~\ref{tc7} below for details.
\begin{proof}[Proof of Theorem~\ref{main7}]
Let $d_{TV}(X,Y)$ denote the total variation distance between the distributions of the random variables $X$ and $Y.$
That is, if $X$ and $Y$ are both non-negative and integer-valued, $d_{TV}(X,Y)=\frac{1}{2}\sum_{n=0}^\infty |P(X=n)-P(Y=n)|.$
By the coupling inequality, $d_{TV}(X,Y)\leq P(X\neq Y).$ Furthermore, there exists a maximal coupling,
that is a random pair $\bigl(\witi X,\witi Y\bigr)$ such that $\witi X$ is distributed the same as $X,$ $\witi Y$ is distributed the same as $Y,$
and $P\bigl(\witi X\neq \witi Y)=d_{TV}(X,Y)$ \cite{thor}.
\par
We will use the following inequalities. For the first claim see, for instance,
Theorem~4 and subsequent Remark~1.1.4 in \cite{poisson}, and for the second one Theorem~1.C(i) in \cite{barbour}.
\begin{lemma}
\label{lemma-approx}
\item[(i)] Let $X=BIN(n,p)$ have the binomial distribution with parameters $n\in\nn,$ $p\in(0,1)$ and $Y$ have
the Poisson distribution with parameter $c=np.$ Then $d_{TV}(X,Y)\leq \frac{p}{2}\cdot \min\{1,c\}.$
\item[(ii)] Let $X$ and $Y$ be two Poisson random variables with parameters $\mu>0$ and $c>\mu,$ respectively.
Then $d_{TV}(X,Y)\leq \min\bigl\{1,\frac{1}{\sqrt{c}}\bigr\}\cdot |\mu-c|.$
\end{lemma}
Using the above results, we can construct a coupling of the Markov chain $(X_k)_{k\in \zz_+}$ and the limiting branching process $(Z_k)_{k\in \zz_+}$ as follows.
The resulting process $\bigl(\witi X_k,\witi Z_k\bigr)_{k\in \zz_+}$ is a Markov chain.
Suppose that the random pairs $\bigl(\witi X_t,\witi Z_t\bigl)$ have been sampled for all $t\leq k$ and that
$\witi X_t=\witi Z_t$ for all $t\leq k.$ Let $i$ be the common value of $\witi Z_k$ and $\witi X_k.$ Sample the next pair $\bigl(\witi X_{k+1},\witi Z_{k+1}\bigr)$
using the maximal coupling for $X_{k+1}$ under the conditional distribution $P(X_{k+1}\in \cdot\,|X_k=i)$ and
$Z^{(c)}_{k+1}$ under the conditional distribution $P\bigl(Z^{(c)}_{k+1}\in \cdot\,\bigl|Z^{(c)}_k=i\bigr).$
After the random time $\tau:=\min\bigl\{k\in\nn:\witi X_k\neq \witi Z_k\bigr\},$ sample $\bigl(\witi Z_t\bigr)_{t\geq \tau}$ and $\bigl(\witi X_t)_{t\geq \tau}$ independently.
Using Lemma~\ref{lemma-approx} and at first approximating $\witi X_{k+1}$ by a Poisson random variable with parameter $(n-i)(1-q^i),$ we obtain that for $c>1,$
\beq
\nonumber
\label{dif}
P\bigl(\witi X_{k+1}\neq \witi Z_{k+1}\bigl|\witi X_k=\witi Z_k=i\bigr)
&
\leq& \frac{1}{2}(1-q^i)+\frac{1}{\sqrt{c i}}|(n-i)(1-q^i)-c i|
\\
&\leq& \frac{ci}{2n}+\frac{1}{\sqrt{c i}}\bigl[ci-(n-i)(1-q^i)\bigr].
\feq
Using the bound in Lemma~\ref{loq} we conclude that in the case when $c>1,$
\beqn
\label{diff}
P\bigl(\witi X_{k+1}\neq \witi Z_{k+1}\bigl|\witi X_k=\witi Z_k=i\bigr)\leq \frac{ci}{2n}+\frac{c^{1/2}i^{3/2}}{2n}+\frac{c^{3/2}i^{3/2}}{2n}
\leq \frac{3c^{3/2}i^{3/2}}{2n}.
\feqn
Similarly, when $c\leq 1,$ without making an assumption on whether $ci\leq 1$ or not,
\beqn
\label{diff1}
P\bigl(\witi X_{k+1}\neq \witi Z_{k+1}\bigl|\witi X_k=\witi Z_k=i\bigr)\leq \frac{ci}{2n}+\frac{ci^2}{2n}+\frac{c^2i^2}{2n}
\leq \frac{3ci^2}{2n}.
\feqn
Recall
\beq
\tau=\inf\bigl\{k\in\nn: \witi X_k\neq \witi Z_k\bigr\}
\feq
and let
\beq
\sigma=\inf\bigl\{k\in \nn: \witi Z_k=0\bigr\}.
\feq
To evaluate the distribution function of $T$ we will use the following inequalities:
\beqn
\nonumber
P(T\leq m)&\leq& P(T\leq m,\tau>m)+P(\tau\leq m)
\\
\label{basic}
&\leq&
 P(\sigma\leq m)+P(\tau\leq m)
\feqn
and
\beqn
\label{basic1}
P(T\leq m)\geq P(\sigma\leq m).
\feqn
The latter inequality holds true because the avalanche process $X$ is stochastically dominated
by the branching process $Z^{(c)}$ by virtue of Proposition~\ref{coupling1}.
\par
The following lemma summarizes results of \cite{agresti74} regarding the distribution (subdistribution if the process is supercritical)
function $P(\sigma \leq m).$ Specific bounds for the extinction time of a Poisson distribution are derived in Theorem~2 of \cite{agresti74}.
\par
Recall $\alpha_c$ from \eqref{alpha} and $s(c),r(c)$ from \eqref{s74}.
\begin{lemma}
[\cite{agresti74}]
\label{tbounds}
\item [(i)] If $c>1,$ then for any $m\in\nn,$  $\bigl(\frac{\alpha_c s_1(1-c^m\alpha_c^m)}{s_1-c^m\alpha_c^m}\bigr)^{i_0}\leq
P(\sigma\leq m)\leq \bigl(\frac{\alpha_c r_1 (1-c^m\alpha_c^m)}{r_1-c^m\alpha_c^m}\bigr)^{i_0},$
where $s_1=s(\alpha_cc)$ and $r_1=r(\alpha_cc).$
\item [(ii)] If $c<1,$ then for any $m\in\nn,$  $\bigl(\frac{s_2(1-c^m)}{s_2-c^m}\bigr)^{i_0}\leq
P(\sigma\leq m)\leq \bigl(\frac{r_2(1-c^m)}{r_2-c^m}\bigr)^{i_0},$
where $s_2=s(c)$ and $r_2=r(c).$
\item [(iii)] If $c=1,$ then for any $m\in\nn,$  $\bigl(\frac{m}{m+2}\bigr)^{i_0}\leq
P(\sigma\leq m)\leq \bigl(\frac{m}{m+e-1}\bigr)^{i_0}.$
\end{lemma}
We next estimate $P(\tau\leq m).$ For $k\in\nn$ let $W_k=\max_{0\leq i \leq k} \witi Z_i.$ By the Markov property of
$\bigl(\witi X_k,\witi Z_k\bigr)_{k\in \zz_+},$ for any $J\in\nn$ we have:
\begin{align}
\nonumber
P(\tau\leq m)&\leq P(\tau\leq m~\mbox{\rm and}~W_m< J)+P(W_m\geq J)
\\
\nonumber
&\leq
P\Bigl(\bigcup_{k=1}^m \bigl\{\witi X_{k-1}=\witi Z_{k-1}<J,\,\witi X_k\neq \witi Z_k\bigr\}\Bigr)+P(W_m\geq J)
\\
\label{i1}
&\leq m\cdot P\bigl(\witi X_k\neq \witi Z_k\,\bigr|\,\witi X_{k-1}=\witi Z_{k-1}<J\bigr) +P(W_m\geq J).
\end{align}
The first part of the next lemma is an improved version of Lemma~5.6 in \cite{jmb}.
\begin{lemma}
\label{lemma1}
\item [(i)] Suppose that $c>1.$ Then there exist constants $\theta >0$ and $K >0$ such that for any $x>0$ and $m\in\nn$ we have
$P(W_m\geq x c^m)\leq K^{i_0}e^{-\theta x}.$
\item [(ii)] If $c<1,$ then for any $m\in\nn$ and integer $J\geq i_0,$ we have $P(W_m\geq J) \leq \frac{\alpha_c^{i_0}-1}{\alpha_c^J-1}.$
\item [(iii)] If $c=1,$ then for any $m\in\nn$ and integer $J\geq i_0,$ we have $P(W_m\geq J) \leq  \frac{i_0}{J}.$
\end{lemma}
\begin{proof}[Proof of Lemma~\ref{lemma1}]
$\mbox{}$
\\
(i) Let $U_k:=c^{-k}\witi Z_k.$ Then $(U_k)_{k\in\zz_+}$ is a martingale with respect to its natural filtration.
For any $\theta>0,$ $f(x)=e^{\theta x}$ is a convex function and hence the sequence $e^{\theta U_k},$ $k\in \zz_+,$ form a submartingale. Hence, by Doob's maximal inequality,
\beq
P(W_m\geq xc^m)&\leq & P\bigl(\max_{0\leq k \leq m} e^{\theta U_k}\geq e^{\theta x}\bigr)\leq
e^{-\theta x}E\bigl(e^{\theta U_m}\bigr).
\feq
The result now follows from Theorem~4 in \cite{kba} which states that $\sup_{m\in\nn} E\bigl(e^{\theta U_m}\bigr)<+\infty$ for some $\theta>0.$
\\
$\mbox{}$
\\
(ii) and (iii) This is a direct implication of Theorem~2 in \cite{lind}.
\end{proof}
The claims in the theorem follow now by combining the bounds in \eqref{diff}--\eqref{i1} with the estimates in Lemmas~\ref{tbounds} and~\ref{lemma1}.
In the case when $c=1,$ the optimization problem over the optimal choice of the parameter $J$ in the upper bound for $P(T\leq m)$ can be solved explicitly,
and we choose $J=\bigl(\frac{n i_0^2}{3m}\bigr)^{1/3}.$
\end{proof}
We will next consider a family of avalanches under Assumption~\ref{assume7}. For $J>0$ let
\beqn
\label{tiveps3}
T^{(n)}=\min\bigl\{k\in\nn:X^{(n)}_k=0\bigr\}.
\feqn
The following result shows that the bounds in Lemma~\ref{tbounds} hold asymptotically for the avalanche model.
\begin{corollary}
\label{th7}
Let Assumption~\ref{assume7} hold.
\item [(i)] If $\lambda>1,$ then for any $m\in\nn,$
\beq
\Bigl(\frac{\alpha_\lambda s_3(1-\lambda^m\alpha_\lambda^m)}{s_3-\lambda^m\alpha_\lambda^m}\Bigr)^{i_0}
&\leq&
\liminf_{n\to\infty}P\bigl(T^{(n)}\leq m\bigr)
\\
&
\leq& \limsup_{n\to\infty}P\bigl(T^{(n)}\leq m\bigr)
\leq \Bigl(\frac{\alpha_\lambda r_3 (1-\lambda^m\alpha_\lambda^m)}{r_3-\lambda^m\alpha_\lambda^m}\Bigr)^{i_0},
\feq
where $s_3=s(\alpha_\lambda\lambda)$ and $r_3=r(\alpha_\lambda\lambda).$
\item [(ii)] If $\lambda<1,$ then for any $m\in\nn,$
\beq
\Bigl(\frac{s_4(1-\lambda^m)}{s_4-\lambda^m}\Bigr)^{i_0}\leq
\liminf_{n\to\infty}P\bigl(T^{(n)}\leq m\bigr)\leq \limsup_{n\to\infty}P\bigl(T^{(n)}\leq m\bigr)
\leq
\Bigl(\frac{r_4(1-\lambda^m)}{r_4-\lambda^m}\Bigr)^{i_0},
\feq
where $s_4=s(\lambda)$ and $r_4=r(\lambda).$
\item [(iii)] If $\lambda=1,$ then for any $m\in\nn,$
\beq
\Bigl(\frac{m}{m+2}\Bigr)^{i_0}\leq
\liminf_{n\to\infty}P\bigl(T^{(n)}\leq m\bigr)\leq \limsup_{n\to\infty}P\bigl(T^{(n)}\leq m\bigr)
\leq \Bigl(\frac{m}{m+e-1}\Bigr)^{i_0}.
\feq
\end{corollary}
\begin{proof}
Informally speaking, the obvious strategy for the proof is to take $n$ to infinity in the conclusions of Theorem~\ref{main7}.
Let $\alpha_1=r(1)=s(1)=1.$ Observe that for any $\lambda>0,$ $\lim_{\mu\to \lambda} \alpha_\mu=\alpha_\lambda,$
and for any $\lambda \in [0,1],$
\beq
\lim_{\mu\to \lambda} s(\mu)=s(\lambda), \quad
\lim_{\mu\to \lambda} r(\mu)=r(\lambda),
\feq
where in the case $\lambda=1$ the latter limits are understood as left limits. Furthermore, the limits in Lemma~\ref{tbounds} are continuous functions of $c.$
This is evident for $c\neq 1,$ and for $c=1$ by the L'Hospital rule we have
\beq
\lim_{c\to 1^-} \frac{s(c)(1-c^m)}{s(c)-c^m}=
\lim_{c\to 1^-}\frac{-(1-c^m)-m(2-c)c^{m-1}}{-1-(m+1)c^m}=\frac{m}{m+2}
\feq
and
\beq
\lim_{c\to 1^-} \frac{r(c)(1-c^m)}{r(c)-c^m}&=&
\lim_{c\to 1^-}\frac{(e^{-c}-ce^{-c})(1-c^m)-me^{-c}c^m}{e^{-c}-ce^{-c}-mc^{m-1}(e^{-c}-1+c)-c^m(-e^{-c}+1)}
\\
&=&\frac{me^{-1}}{me^{-1}-e^{-1}+1}=\frac{m}{m+e-1},
\feq
as desired.
\par
For $k\in\nn$ and $\mu>0,$ let $W^{(\mu)}_k=\max_{0\leq i \leq k} Z^{(\mu)}_i.$
In comparison with the proof of Theorem~\ref{main7}, the only nuance here is that before letting $n$ go to infinity the term
$P\bigl(W^{(np_n)}_m\geq J\bigr)$ in \eqref{i1}, which a priori is not uniform in $n,$  should be replaced with $P\bigl(W^{(c)}_m\geq J_n\bigr),$
where $c$ is the constant introduced in Condition~\ref{cdla} and $J_n$ is a suitable sequence which in particular satisfies $\lim_{n\to\infty}J_n=+\infty.$
\footnote{Technically, this part of the proof makes the claim a corollary to a slight modification of Theorem~\ref{main7} rather than to the theorem
itself.}
\end{proof}
Another interesting consequence of Theorem~\ref{main7} is the following result which shows that on the right scale,
the asymptotic behavior of the avalanche model coincides with that of the branching process (see, for instance, compact introduction section
in \cite{cks} for the underlying branching process estimates).
\begin{corollary}
\label{tc7}
Let Assumption~\ref{assume7} hold.
\item [(i)] Suppose that $\lambda>1.$   Let $(\beta_n)_{n\in\nn}$ be a sequence of positive constants such that
\beq
\lim_{n\to\infty} \beta_n\log n=+\infty\qquad \mbox{\rm and}\qquad \limsup_{n\to\infty} \beta_n<\frac{2}{3\log \lambda}.
\feq
Let $m_n=\beta_n\log n.$ Then $\lim_{n\to\infty}P\bigl(T^{(n)}> m_n\bigr)=1-\alpha_\lambda^{i_0}.$
\item [(ii)] Suppose that $\lambda>1.$ Let $(\beta_n)_{n\in\nn}$ be a sequence of positive constants such that
\beq
\lim_{n\to\infty} \frac{\beta_n(\log n)^3}{n}=0\qquad \mbox{\rm and}\qquad \liminf_{n\to\infty} \beta_n>-\frac{1}{3\log \lambda}.
\feq
Let $m_n=\beta_n\log n.$ Then $\lim_{n\to\infty}\frac{1}{m_n}\log P\bigl(T^{(n)}>m_n\bigr)=\log \lambda.$
\item [(iii)] Assume $\lambda=1.$ Let $(m_n)_{n\in\nn}$ be a sequence of positive integers such that $\lim_{n\to\infty} m_n=+\infty$ and
$\lim_{n\to\infty} \frac{m_n^4}{n}=0.$ Then $\lim_{n\to\infty} m_n P\bigl(T^{(n)}>m_n\bigr)=2i_0.$
\end{corollary}
\begin{proof}
In each of the three cases ($\lambda>,<,=1$) the sequence $\beta_n$ is chosen in such a way that the term corresponding to $P(W_n\geq m)$
in the key estimate \eqref{i1} and consequently in the conclusions of Theorem~\ref{main7} is dominated by the bounds
for $P(\sigma\leq m)$ which are supplied by Lemma~\ref{tbounds}. As the result,
the term vanishes as $n$ goes to infinity, and the asymptotic behavior of $P\bigl(T^{(n)}>m_n\bigr)$ coincides with its counterpart for branching processes.
The only subtlety is in the proof of part (iii), where the asymptotic behavior exhibited by the lower and upper bounds in Lemma~\ref{tbounds} do not match,
namely $\lim_{n\to\infty} m\bigl\{1-\bigl(\frac{m}{m+2}\bigr)^{i_0}\bigr\}=2i_0$ while
$\lim_{n\to\infty} m\bigl\{1-\bigl(\frac{m}{m+e-1}\bigr)^{i_0}\bigr\}=(e-1)i_0.$ However,
Kolmogorov's estimate (see, for instance, Theorem~C in \cite{cks}) ensures that $\lim_{m\to\infty} P(\sigma>m)=2i_0.$ In view of \eqref{i1}
this yields the desired result in the case $\lambda=1.$
\end{proof}
\section{Deterministic approximation}
\label{dapprox}
The branching approximation reflects the dynamics of the process suitably when the latter is subcritical and the size
of avalanche decays exponentially fast, similarly to the branching approximation process.  In contrast, in the supercritical regime,
namely when Condition~\ref{assume5} is satisfied with $c>1$ or Assumption~\ref{assume7} holds with $\lambda>1,$ the results in Section~\ref{funda}
tend to provide a little or no information on the asymptotic behavior of the model, and even a partial improvement of this situation is clearly desirable.
However, even in the supercritical case, one should expect that with $\veps \ll 1,$ the branching approximation is adequate for $\sim \veps \log n$
of first steps while the ratio $\frac{X_k}{n}$ remains small, cf. Section~6.3.1 of \cite{dgbook}. On the other hand, once the ratio becomes of order one,
one can expect that the dynamics of the model will be almost deterministic and follow a ``mean-field" equation due to the law of large numbers.
The section is devoted to the study of this deterministic approximation of the model. The relevance of this regime to a supercrtical avalanche
model is formally elucidated by the results in Section~\ref{fixp}.
\par
The rest of the section is divided into three subsections.
Section~\ref{msns} discusses an heuristic computation based
on a complete decoupling of a week dependence between the network nodes, that shades an additional light on the nature of the deterministic approximation.
In Section~\ref{large} the deterministic dynamical system is formally introduced and studied. Interestingly enough, the underlying dynamics resembles
but is richer than that of the logistic equation. This was already observed in \cite{cooke, longini}. Section~\ref{comp} addresses the question for how long
the trajectories of an avalanche Markov chain and the corresponding deterministic dynamical system
remain close each to other, provided they began at the same point.
\subsection{An extra motivation and a precise comparison result}
\label{msns}
Our motivation in this section is partially coming from the following heuristic computation, which is an adaptation to our framework of the one proposed in \cite{g1}
for a model of avalanches in a Boolean network. Consider an avalanche Markov chain $(X_k)_{k\in\zz_+}$ with transition kernel introduced in \eqref{ker},
and recall $\cale_k$ from \eqref{wk} and $\xi_k(x)$ from \eqref{xik}.
If the random variables $\{\cale_k(x):x\in V_n\}$ were independent and identically distributed for some $k\in\zz_+,$
we could relate their common value $\xi_k$ to the common value $\xi_{k+1}$ of $\xi_{k+1}(x),$ $x\in V_n,$ by means of the following recursion \cite{g1}:
\beqn
\nonumber
\xi_{k+1}&=&\bigl(1-\xi_k\bigr)\sum_{i=0}^{n-1}\binom{n-1}{i}\xi_k^i(1-\xi_k)^{n-1-i}\bigl(1-q^{i}\bigr)
\\
\label{hr1}
&=&
(1-\xi_k)\bigl[1-(q\xi_k+1-\xi_k)^n\bigr]=(1-\xi_k)\bigl[1-(1-p\xi_k)^n\bigr].
\feqn
Consequently, the following inequality would hold:
\beqn
\label{h1}
\xi_{k+1}\geq (1-\xi_k)\bigl(1-e^{-pn\xi_k}\bigr)= (1-\xi_k)\bigl(1-e^{-c\xi_k}\bigr),
\feqn
where we denote $c=np.$ We remark that similar decoupling arguments are often used in a physics literature to justify a
mean-field approximation in a complex locally tree-like network (hence a fairly weak dependence between the nodes),
see, for instance, \cite{ca,b6,iva,a, a1}.
\par
Of course, $\cale_k(x)$ are not independent and in general are not identically distributed (the latter depends on whether the distribution of $X_0$
is exchangeable or not, cf. Proposition~\ref{mixing}). For a different, but somewhat related model of avalanches, it is argued in \cite{g1} (see the comment [4] in the References section there)
that the above i.\,i.\,d. assumption ``is true for large networks with well-behaved degree distributions, but excludes networks with
hubs which output to a substantial fraction of the nodes." It is of interest to note that, in agreement with this general heuristic principle,
a suitable modification of \eqref{hr1} and, consequently, of \eqref{h1} serve in a certain rigorous sense as a good approximation
for the avalanche model. This is accomplished in \eqref{vpsiphi} below.
\par
We will now proceed with a rigorous modification of the above heuristic calculation. For $\alpha>0$ let
\beq
g_\alpha(x)=(1-x)(1-e^{-\alpha x}),\qquad x\in [0,1],
\feq
and
\beq
\chi_\alpha=\max_{x\in(0,1)}g_\alpha(x),\quad \mbox{\rm and}\quad
\nu_\alpha=\underset{x\in(0,1)}{\mbox{\rm argmax}}~g_\alpha(x).
\feq
Note that $\nu_\alpha$ is uniquely defined for all $\alpha>0.$
Further, if $\alpha>1,$ let $\zeta_\alpha$ be the unique on $(0,1)$ solution to the fixed point equation
\beq
g_\alpha(\zeta_\alpha)=\zeta_\alpha, \qquad \zeta_\alpha\in(0,1).
\feq
Remark that $\zeta_\alpha=g_\alpha(\zeta_\alpha)<1-\zeta_\alpha$ implies that $\zeta_\alpha<1/2.$ The following lemma summarizes other basic
properties of the function $g_\alpha$ that we are going to use (see also Theorem~\ref{cooke} in Section~\ref{large} below).
The proof of the lemma is omitted, the dependence of the graph of $g_\alpha$ on the parameter $\alpha$ as well as
basic properties of $g_\alpha$ are illustrated in Fig.~\ref{fig1} below. For more details see,
for instance, \cite{cooke} where the function $g_\alpha$ is studied systematically in a similar context.
\begin{lemma}
\label{cri}
\item [(i)] $g_\alpha(x)$ is increasing on $(0,\nu_\alpha)$ and decreasing on $(\nu_\alpha,1).$
\item [(ii)] If $\alpha\leq 1,$ then $g_\alpha(x)<x$ for all $x\in (0,1).$
\item [(iii)] If $\alpha>1,$ then $g_\alpha(x)>x$ on $(0,\zeta_\alpha)$ and $g_\alpha(x)<x$ on $(\zeta_\alpha,1).$
\item [(iv)] $g_\alpha(x)<\alpha x$ for all $\alpha>0$ and $x\in (0,1).$
\item [(v)] There exists a transitional value $\alpha_{\rm tr}>1$ such that $\nu_\alpha>\zeta_\alpha$ if and only if $\alpha<\alpha_{\rm tr}.$
\end{lemma}
Numerical simulations indicate that $\alpha_{\rm tr}\approx 2.46742.$ The role of the threshold in the dynamics of the model
is discussed in more detail in \cite{cooke} and \cite{duration}.

\begin{figure}[!ht]
\begin{tikzpicture}[scale=1.5]
\begin{axis}[my style, xtick={0,...,1}, ytick={0,...,1},
xmin=0, xmax=1.3, ymin=0, ymax=1.3]
\addplot[semithick][domain=0:1,samples=100]{(1-x)*(1-exp(-1.8 *x))};
\addplot[very thick][domain=0:1,samples=100]{(1-x)*(1-exp(-1 *x))};
\addplot[semithick][domain=0:1,samples=100]{(1-x)*(1-exp(-0.3*x))};
\addplot[semithick][domain=0:1,samples=100]{(1-x)*(1-exp(-5*x))};
\addplot[semithick][domain=0:1,samples=500]{(1-x)*(1-exp(-12*x))};
\addplot[very thick][domain=0:1,samples=100]{(1-x)*(1-exp(- 2.46742*x))};
\addplot[dashed][domain=0:1]{x};
\addplot[dashed][domain=0:1]{1-x};
\node at (0.43,0.26) {\tiny $\alpha=1.8$};
\node at (0.45,0.155) {\tiny $\alpha=1$};
\node at (0.51,0.029) {\tiny $\alpha=0.3$};
\node at (0.27,0.41) {\tiny $\alpha_{\rm tr}$};
\node at (0.18,0.575) {\tiny $\alpha=5$};
\node at (0.1,0.77) {\tiny $\alpha=12$};
\end{axis}
\end{tikzpicture}
\caption{Graph of the function $g_\alpha(x)=(1-x)(1-e^{-\alpha x})$
for several values of the parameter $\alpha,$ including the critical branching value $\alpha=1,$ for which $g_1'(0)=1,$ and
$\alpha=2.46742$ which is a close approximation to the transitional value $\alpha_{\rm tr}$. \label{fig1}}
\end{figure}
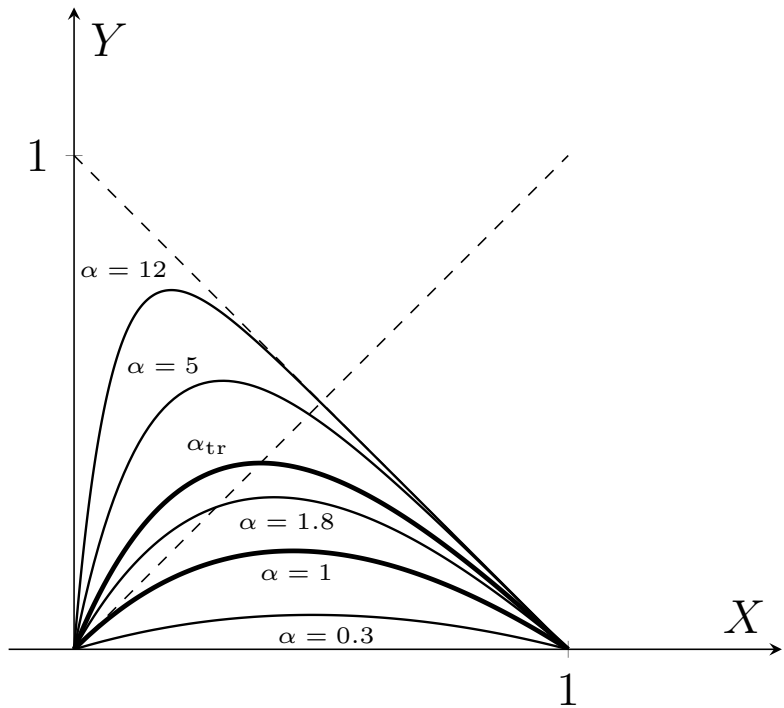

For an avalanche model with transition kernel \eqref{ker} and fixed network size $n\in\nn,$ let
\beqn
\label{vphi1}
\varphi_k=\frac{1}{n}E(X_k), \qquad k\in\zz_+.
\feqn
Note that if the distribution of $X_0$ is exchangeable (i.\,e. invariant with respect to permutation of the network nodes),
then $\varphi_k=\xi_k(x),$ where $\xi_k$ is defined in \eqref{xik}, for all $k\in\zz_+$ and $x\in V_n.$ By the bounded convergence theorem,
for any fixed $n\in\nn,$ $\lim_{k\to\infty} \varphi_k=0$ regardless of the choice of parameter $p>0.$
We have:
\begin{theorem}
\label{prop}
Let $(X_k)_{k\in\zz_+}$ be an avalanche model with transition kernel \eqref{ker}. Denote
\beqn
\label{al}
\alpha=-n\log(1-p).
\feqn
The following holds true:
\item[(i)]  Let $(\phi_k)_{k\in\zz_+}$ be defined recursively by setting $\phi_0=\frac{1}{n}E(X_0)$ and
\beq
\phi_{k+1}=1-e^{-\alpha \phi_k}.
\feq
Then $\varphi_k\leq \phi_k$ for all $k\in\zz_+.$
\item[(ii)] $\sup_{k\in\zz_+} \varphi_k\leq \max\bigl\{\phi_0,\chi_\alpha\bigr\}.$
\item [(iii)] Let $(\psi_k)_{k\in\zz_+}$ be defined recursively by setting $\psi_0=\frac{1}{n}E(X_0)$ and
\beq
\psi_{k+1}=g_\alpha (\psi_k).
\feq
If $\alpha\leq \alpha_{\rm tr}$ and
\beq
\varphi_0\leq \zeta_\alpha,
\feq
then
\beq
\varphi_k\leq \psi_k \qquad  \forall\,k\in\zz_+,
\feq
and, moreover,
\beq
\sup_{k\in\zz_+} \varphi_k\leq \zeta_\alpha \leq \chi_\alpha.
\feq
\end{theorem}
Note that in view of the fact that $1-e^{-\alpha \phi_k}<\alpha \phi_k$ and the result in part~(iv) of Lemma~\ref{cri}, the theorem improves
the result in the basic Corollary~\ref{avec} even in the subcritical case $\alpha<1.$ The relevance of the phase transition at $\alpha_{\rm tr}$
to the dynamics of the avalanche Markov chain is further discussed in Section~\ref{large} below, at the paragraph following Corollary~\ref{propc}.
\begin{proof}[Proof of Theorem~\ref{prop}]
\item [(i)] It follows from \eqref{ker}, \eqref{al}, and Jensen's inequality which is applied in the last step to the concave on $(0,1)$
function $g(x)=(1-x)(1-e^{-\alpha x})$ that
\beqn
\nonumber
\varphi_{k+1}(x)&=&E\Bigl[\Bigl(1-\frac{X_k}{n}\Bigr)\bigl(1-q^{X_k}\bigr)\Bigr]
=
E\Bigl[\Bigl(1-\frac{X_k}{n}\Bigr)\Bigl(1-\bigl(1-p\bigr)^{X_k}\Bigr)\Bigr]
\\
\nonumber
&=&
E\Bigl[\Bigl(1-\frac{X_k}{n}\Bigr)\Bigl(1-\Bigl\{\bigl(1-p\bigr)^{p}\Bigr\}^{\frac{X_k}{p}}\Bigr)\Bigr]
\\
\label{phi}
&=&
E\Bigl[\Bigl(1-\frac{X_k}{n}\Bigr)\Bigl(1-e^{-\frac{\alpha X_k}{n}}\Bigr)\Bigr]
\leq (1-\varphi_k)(1-e^{-\alpha \varphi_k}).
\feqn
In particular, $\varphi_{k+1}\leq  1-e^{-\alpha \varphi_k}.$ Since $f(x)=1-e^{-\alpha x}$ is a monotone increasing function,
an induction argument shows that $\varphi_k\leq \phi_k$ for all $k\in\zz_+.$
\\
$\mbox{}$
\item [(ii)] The claim is immediate from \eqref{phi}.
\\
$\mbox{}$
\item [(iii)] Observe that $g_\alpha$ is monotone increasing on $(0,\zeta_\alpha)$ for $\alpha\leq \alpha_{\rm tr}.$
Therefore, it follows from \eqref{phi} by using induction on $k,$ that $\varphi_k\leq \psi_k\leq \zeta_\alpha$ for all $k\in\zz_+.$
Moreover, if $\alpha\leq 1$ then $\psi(x)<x$ on $(0,\zeta_\alpha),$ and hence $\psi_{k+1}\leq \psi_k$ for all $k\in\zz_+,$ while
if $\alpha\in (1,\alpha_{\rm tr}]$ then $\psi(x)>x$ on $(0,\zeta_\alpha),$ and hence $\psi_{k+1}\geq \psi_k$ for all $k\in\zz_+.$
Since $g_\alpha$ has either one fixed point at zero (for $\alpha\leq 1$) or two at zero and $\zeta_\alpha$ (for $\alpha>1$),
this implies that $\lim_{k\to\infty} \varphi_k=0$ for $\alpha\leq 1$ while
$\sup_{k\in\zz_+} \varphi_k\leq   \sup_{k\in\zz_+} \psi_k\leq \lim_{k\to\infty} \varphi_k=\zeta_\alpha$ in the case that $\alpha \in (1,\alpha_{\rm tr}].$
\end{proof}
\subsection{Deterministic approximation for large size networks}
\label{large}
We turn now to a study of an ensemble of avalanche models which satisfies Assumption~\ref{assume7}. First we discuss a direct implication
of the results in Theorem~\ref{prop} for the asymptotic behavior of the expected fraction of excited nodes in a network of the ensemble.
The main results of this section are stated afterwards in Theorems~\ref{da} (regarding asymptotic behavior of $\frac{1}{n}\xn_k$ for large $n$) and
Corollary~\ref{h} (a consequence of the results in Theorem~\ref{da} for the heterogeneity of a network in the ensemble).
\par
First, observe that continuity of the results in Theorem~\ref{prop} in the parameters $\alpha$ and the initial data $\varphi_0$ together with the monotonicity
of $g_\alpha$ on the interval $\bigl(0,\zeta_\alpha\bigr)$ lead to the following corollary to the theorem.
\begin{corollary}
\label{propc}
Let Assumption~\ref{assume7} hold, and suppose that the following limit exists and belongs to $(0,1):$
\beq
\veps_0=\lim_{n\to\infty} \frac{1}{n}E\bigl(X^{(n)}_0\bigr).
\feq
For $n\in\nn$ and $k\in\zz_+$ let
\beqn
\label{vphi}
\varphi_{n,k}=\frac{1}{n}E\bigl(\xn_k\bigr).
\feqn
be the counterpart of $\varphi_k$ introduced for a single network in \eqref{vphi1}.
\par
Then the following holds true:
\item[(i)]  Let $(\phi_k)_{k\in\zz_+}$ be defined recursively by setting $\phi_0=\veps_0$ and
\beq
\phi_{k+1}=1-e^{-\lambda \phi_k}.
\feq
Then $\limsup_{n\to\infty}\varphi_{n,k}\leq \phi_k$ for all $k\in\zz_+.$
\item[(ii)] $\limsup_{n\to\infty} \varphi_{n,k}\leq \chi_\lambda$ for all $k\in\nn.$
\item [(iii)] Recall $\alpha_{\rm tr}$ from Lemma~\ref{cri}. Let $(\psi_k)_{k\in\zz_+}$ be defined recursively by setting
\beqn
\label{psi}
\psi_0=\veps_0\qquad \mbox{\rm and}\qquad \psi_{k+1}=g_\lambda (\psi_k).
\feqn
If $\lambda<\alpha_{\rm tr}$ and $\veps_0<\zeta_\lambda,$  then
\beq
\limsup_{n\to\infty}\varphi_{n,k}\leq \psi_k\qquad \mbox{\rm and}\qquad
\limsup_{n\to\infty} \varphi_{n,k}\leq \zeta_\lambda \leq \chi_\lambda
\feq
for all $k\in\zz_+.$
\end{corollary}
Let
\beqn
\label{xnk}
x_{n,k}=\frac{X^{(n)}_k}{n},\qquad n\in\nn,\,k\in\zz_+.
\feqn
Under Assumption~\ref{assume7}, define the asymptotic branching factor as a function $b:(0,1)\to\rr$ by setting
\beq
b(x)=\lim_{n\to\infty} \frac{1}{x}E(x_{n,k+1}\,|\,x_{n,k}=x)=\frac{g_\lambda(x)}{x},\qquad x\in (0,1).
\feq
It turns out (see \cite{cooke} for details) that the behavior of the sequence $b(\psi_k)$ for $\alpha\in(1,\alpha_{\rm tr})$
and $\alpha>\alpha_{\rm tr}$ differ qualitatively, that is a secondary phase transition in the avalanche model occurs
at the transitional value $\alpha_{\rm tr}.$ For instance, if $\alpha\in (1,\alpha_{\rm tr})$ and $\psi_0$ is sufficiently small, then $b(\psi_k)$ increases monotonically to one as $k\to\infty.$
In contrast, that's not necessarily true when $\alpha>\alpha_{\rm tr}.$ For example, in the case $\alpha>\alpha_{\rm tr},$ if for some $m\in\nn,$ $\psi_0<\nu_\alpha$ and $\psi_m=\nu_m,$
then the sequence $\psi_k$ increases monotonically at the first $m$ steps and then converging to one by oscillating consequently between values which are larger and smaller than one \cite[p.~81]{cooke}.
Note that by choosing $m$ sufficiently large, we can place $\psi_0$ as close to zero (the only fixed point of the equation $x=g_\alpha(x)$ other than $\zeta_\alpha$) as we wish.
\par
Let $\overset{P}{\to}$ denote convergence in probability as the size of the network $n$ goes to infinity.
The following theorem is an adaptation to our setup of Theorems~1 and~3 in \cite{bdet}
(see also \cite{knerman} for earlier similar results).
\begin{theorem}
\label{da}
Let Assumption~\ref{assume7} hold. Recall $x_{n,k}$ from \eqref{xnk} and suppose that
\beq
x_{n,0}\overset{P}{\to}\psi_0
\feq
for some constant $\psi_0\in (0,1).$ Then the following holds true:
\item [(a)]  $x_{n,k}\overset{P}{\to} \psi_k$ for all $k\in\zz_+,$ where $\psi_{k+1}=g_\lambda(\psi_k).$
\item [(b)] Let
\beq
y_{n,k}=\sqrt{n}(x_{n,k}-\psi_k),\qquad n\in\nn,\,k\in\zz_+
\feq
and set
\beq
v(x):=g_\lambda(x)e^{-\lambda x}=(1-x)e^{-\lambda x}(1-e^{-\lambda x}),\qquad x\in [0,1].
\feq
Suppose that in addition to \eqref{xnk},  $y_{n,0}$ converges weakly, as $n$ goes to infinity, to some (possibly random) $Y_0.$
Then the sequence $y^{(n)}:=(y_{n,k})_{k\in\zz_+}$ converges in distribution, as $n$ goes to infinity,
to a time-inhomogeneous Gaussian $AR(1)$ sequence $(Y_k)_{k\in\zz_+}$ defined by
\beqn
\label{ar1}
Y_{k+1}=g'_\lambda(\psi_k)Y_k+e_k,
\feqn
where $e_k,$ $k\in\zz_+,$ are  independent Gaussian variables, each $e_k$ distributed as $N\bigl(0,v(\psi_k)\bigr).$
\end{theorem}
\begin{proof}
Let $\mz:=\bigl\{\mz_{k,i}^{(n,j)}:n,i,j\in\nn,k\in\zz_+\bigr\}$
be a collection of independent Bernoulli variables with
\beq
P\bigl(\mz_{k,i}^{(n,j)}=1\bigr)=1-q_n^j\qquad \mbox{\rm and}\qquad P\bigl(\mz_{k,i}^{(n,j)}=1\bigr)=q_n^j.
\feq
Thus, without loss of generality, we can assume that
\beq
\xn_{k+1}=\sum_{i=1}^{n-\xn_k} \mz_{k,i}^{(n,\xn_k)}.
\feq
Let $\my:=\bigl\{\my_{k,i}^{(n,j)}:n,i,j\in\nn,k\in\zz_+\bigr\}$ be another collection of independent Bernoulli variables
defined on the same probability space, and such that
\beq
P\bigl(\my_{k,i}^{(n,j)}=1\bigr)=1-e^{-\frac{\lambda j}{n}}\qquad \mbox{\rm and}\qquad P\bigl(\my_{k,i}^{(n,j)}=1\bigr)=e^{-\frac{\lambda j}{n}}.
\feq
For $n\in\nn,$ let $\alpha_n=np_n.$ By using the maximal coupling for two Bernoulli variables, we can and will assume that the
pairs $\bigl(\mz_{k,i}^{(n,j)},\my_{k,i}^{(n,j)}\bigr)$ are
independent $\{0,1\}^2$-random variables, and
\beq
P\bigl(\mz_{k,i}\neq \my_{k,i}\bigr)=\Bigl|e^{-\frac{\lambda j}{ n}}-\Bigl(1-\frac{\alpha_n}{n}\Bigr)^j\Bigr|.
\feq
For $n\in\nn,$ define a new sequence ${\witi X}^{(n)}=\bigl({\witi X}^{(n)}_k\bigr)_{k\in\zz_+}$ by setting ${\witi X}^{(n)}_0=\psi_0$
and
\beq
{\witi X}^{(n)}_{k+1}=\sum_{i=1}^{n-{\witi X}^{(n)}_k} \my_{k,i}^{(n,{\witi X}^{(n)}_k)}.
\feq
Theorems~1 and~3 in \cite{bdet} ensures that the results in the theorem, both the LLN and CLT, hold if we replace $\xn$ with ${\witi X}^{(n)}.$
Thus, in order to prove the theorem, it suffices to show that $\bigl(\xn_k-{\witi X}_k^{(n)}\bigr)\overset{P}{\to} 0$ for all $k\in \zz_+$
(see, for instance, Remark~(i) in \cite[p.~60]{bdet} and/or the last paragraph in the proof of Theorem~3 there, which
both assert that the main result of \cite{karr} goes through to a non-homogeneous chain setting,
and therefore the weak convergence in question is implied by the convergence of transition kernels). To this end, observe that
\beq
E\Bigl( \frac{1}{n}\Bigl|\xn_{k+1}-{\witi X}_{k+1}^{(n)}\Bigr|\Bigr)
&\leq&
\frac{1}{n}E\Bigl(\sum_{i=1}^{n-\xn_k} \Bigl|\mz_{k,i}^{(n,\xn_k)}-\my_{k,i}^{(n,\xn_k)}\Bigr|\Bigr)
\leq E\Bigl(\Bigl|\mz_{k,1}^{(n,\xn_k)}-\my_{k,1}^{(n,\xn_k)}\Bigr|\Bigr)
\\
&=&
P\Bigl(\mz_{k,1}^{(n,\xn_k)}\neq \my_{k,1}^{(n,\xn_k)}\Bigr)
=E\Bigl(\Bigl|e^{-\frac{\lambda \xn_k}{ n}}-\Bigl(1-\frac{\alpha_n}{n}\Bigr)^{\xn_k}\Bigr|\Bigr),
\feq
and hence the claim can be proved by induction, using the bounded convergence theorem.
\end{proof}
Part (a) of the above result and the bounded convergence theorem imply that
\beqn
\label{vpsiphi}
\lim_{n\to\infty} \varphi_{n,k}=\psi_k\qquad \forall\,k\in\zz_+,
\feqn
where $\varphi_{n,k}$ is defined in \eqref{vphi1}. Note that this limit identity is
a reminiscent of the heuristic \eqref{hr1} and \eqref{h1} in our framework.
\par
It is not hard to prove (cf. Remark (v) on p.~61 of \cite{bdet}, see also \cite{knerman})
that if $\psi_0\overset{P}{\to} \zeta_\lambda$ in the statement of Theorem~\ref{da}
and, in addition, $y_{n,0}$ converges weakly, as $n$ goes to infinity, to some $Y_0,$ then the linear recursion \eqref{ar1} can be replaced with
\beq
Y_{k+1}=g'_\lambda(\zeta_\lambda)Y_k+\witi e_k,
\feq
where $\witi e_k,$ $k\in\zz_+,$ are i.\,i.\,d. Gaussian variables, each $\witi e_k$ distributed as $N\bigl(0,v(\zeta_\lambda)\bigr).$
One then can show (see the proof of Theorem~\ref{oth} below) that $|g_\lambda'(\zeta_\lambda)|<1,$ and hence Markov chain $Y_k$ has
a stationary distribution, see \cite[p.~61]{bdet} for more details.
\par
Recall $H_k$ from \eqref{het} and define a normalized heterogeneity $h_{n,k}$ by
\beq
h_{n,k}=\frac{H_{n,k}}{2n(n-1)},\qquad n\in\nn,k\in\zz_+.
\feq
Notice that if the distribution of $\xn_0$ is exchangeable, then $x_{n,k}$ is a probability that two nodes in the network generating $\xn_k$
randomly chosen at time $k$ have different types. The following is immediate from Theorem~\ref{da}.
\begin{corollary}
\label{h}
Under the conditions of Theorem~\ref{da}, the following holds true:
\item [(a)] For $x\in [0,1]$ let $r(x)=\frac{1}{2}x(1-x).$  Then $h_{n,k}\overset{P}{\to} r(\psi_k)$ for all $k\in\zz_+.$
\item [(b)] Let
\beq
{\witi y}_{n,k}=\sqrt{n}\bigl(h_{n,k}-r(\psi_k)\bigr),\qquad n\in\nn,\,k\in\zz_+.
\feq
Suppose that in addition to \eqref{xnk},  $y_{n,0}=\sqrt{n}\bigl(x_{n,0}-\psi_0\bigr)$ converges weakly, as $n$ goes to infinity, to some (possibly random) $Y_0.$
Then the sequence ${\witi y}^{(n)}:=({\witi y}_{n,k})_{k\in\zz_+}$ converges in distribution, as $n$ goes to infinity,
to a time-inhomogeneous Gaussian $AR(1)$ sequence $(\witi Y_k)_{k\in\zz_+},$ where $\witi Y_k=\frac{1}{2}(1-2\psi_k)Y_k$ and $Y_k$ is defined in \eqref{ar1}.
\end{corollary}
\begin{proof}
The convergence in probability of $h_{n,k}$ follows from the continuous mapping theorem and the result in part (a) of Theorem~\ref{da}.
To prove the weak convergence of ${\witi y}_{n,k},$ write
\beq
{\witi y}_{n,k}=\sqrt{n}\bigl(r(x_{n,k})-r(\psi_k)\bigr)+\frac{\sqrt{n}}{2}X_{n,k}(n-X_{n,k})\Bigl(\frac{1}{n(n-1)}-\frac{1}{n^2}\Bigr)
\feq
and observe that
\beq
\Bigl|\frac{\sqrt{n}}{2}X_{n,k}(n-X_{n,k})\Bigl(\frac{1}{n(n-1)}-\frac{1}{n^2}\Bigr)\Bigr|\leq \frac{\sqrt{n}}{2(n-1)} \longrightarrow_{n\to\infty} 0.
\feq
Furthermore, by the mean value theorem,
\beq
\sqrt{n}\bigl(r(x_{n,k})-r(\psi_k)\bigr)=r'(x^*_{n,k})\sqrt{n}(x_{n,k}-\psi_k)=(1-2x_{n,k})\sqrt{n}(x_{n,k}-\psi_k)
\feq
for some $x^*_{n,k}$ between $x_{n,k}$ and $\psi_k.$ In view of the result in part (b) of Theorem~\ref{da}, the claim in part (b) of this theorem
follows now by another application of the continuous mapping theorem.
\end{proof}
The discrete-time dynamical system $\psi_k$ is studied in details in \cite{cooke}.
In particular, the following  result is proved there (Theorem~1 in \cite{cooke}):
\begin{theorem}[\cite{cooke}]
\label{cooke}
Let $\lambda>0$ and $\veps_0\in(0,1)$ be given, and the sequence $(\psi_k)_{k\in\zz_+}$ is defined as in \eqref{psi}. Then the following holds true:
\item [(i)] If $\lambda\in (0,1],$ then $\lim_{k\to\infty} \psi_k=0$ for any $\psi_0\in [0,1].$
\item [(ii)] If $\lambda>1,$ then for all $\psi_0\in [0,1],$ $\lim_{k\to\infty} \psi_k=\zeta_\lambda,$ where $\zeta_\lambda\in (0,1/2)$ is  the unique positive solution to
the fixed point equation $g_\lambda(x)=x.$
\end{theorem}
Together with the results in Section~\ref{fixp} and Theorem~\ref{da}, this theorem implies that when $n$ is large, with high probability,
a supercritical Markov chain $\xn$ will be eventually trapped for a long time in a neighborhood of the global stable point $\zeta_\lambda$
of the map $g_\lambda.$ The next section is devoted to the proof of a certain qualitative form of this informal observation.
\subsection{Comparison of the stochastic and deterministic trajectories}
\label{comp}
The main results of this section are stated in Theorem~\ref{oth} (supercritical case) and Theorem~\ref{oth1} (critical and subcritical case).
\par
Theorem~\ref{da} suggests that when both $n$ and the first generation $\xn_0$ in the avalanche process are substantially large,
the trajectory of the deterministic sequence $\psi_k$ can serve
as a good approximation to the path of the Markov chain $\xn_k.$ Note however that, at least for a supercritical process,
two trajectories cannot in principle stay close each to other forever
since while the latter converges to zero with probability one, the former tends to a non-zero limit by virtue of Theorem~\ref{cooke}.
\par
The following theorem offers some insight into the duration of the time when the deterministic and the stochastic paths stay fairly close each to other,
before they become significantly separated each from another at the first time. The theorem is a suitable modification of some results in \cite{ozgur}. In words, the theorem asserts
that if Assumption~\ref{assume7} holds with $\lambda>1$ and the scaled initial population $x_{n,0}$ is close enough to the stable point
of the map $g_\lambda$ and $n$ is large, the trajectory of the avalanche model will stay close to the deterministic sequence $\psi_k$
for a time which is exponentially large in the network size $n.$ For reader's convenience we give a short detailed proof of the theorem
which generally follows the line of argument in \cite{ozgur} but is different in several details.
\begin{theorem}
\label{oth}
Suppose that Assumption~\ref{assume7} is satisfied with $\lambda>1.$ For $\delta>0,$ let
\beqn
\label{taun}
\tau_n(\delta)=\inf\bigl\{k\in\zz_+:|x_{n,k}-\psi_k|\geq \delta\bigr\},
\feqn
where the sequence $\psi_k$ is defined in \eqref{psi}.
\par
There exist an interval $(a,b)\subset (0,1)$ including
$\zeta_\lambda$ and constants $\gamma >0,$ $\delta_0\in (0,1),$ and $n_0\in\nn$ such that if
$x_{n,0}\in (a+\delta_0,b-\delta_0),$ and $n>n_0,$ then for any $m\in\nn$ and $\delta \in (0,\delta_0)$ we have
\beq
P(\tau_n(\delta)>m)\geq \bigl(1-2e^{-\gamma \delta^2 n}\bigr)^m
\geq 1-2me^{-\gamma \delta^2 n}.
\feq
Furthermore, the constants $a,b, \delta_0$ and $\gamma$ depend on the sequence of parameters $(p_n)_{n\in\nn}$ through $\lambda$ only
(this is not necessarily true for $n_0,$ which in general is not exclusively determined by the value of $\lambda$).
\end{theorem}
\begin{proof}
Pick first $b\in (0,1)$ and then $a\in (0,1)$ in such that a manner that
\beq
0<a<\min\{\nu_\lambda,\zeta_\lambda,g_\lambda(b)\}<\max\{\nu_\lambda,\zeta_\lambda\}<b<1.
\feq
Let $I=(a,b)$ and $h=\min\{g_\lambda(a),g_\lambda(b)\}.$ Then
\beq
a<h,\qquad\nu_\lambda<b,\qquad g(I)\subset (h,\nu_\lambda).
\feq
Thus, if we set
\beq
\veps=\min\bigl\{b-\zeta_\lambda,h-a,(b-a)/2\bigr\},
\feq
we get
\beqn
\label{abe}
g(a,b)\subset (a+\veps,b-\veps)\qquad \mbox{\rm and}\qquad \zeta_\lambda\in (a+\veps,b-\veps).
\feqn
The latter assertion is true because $g_\lambda(\zeta_\lambda)=\zeta_\lambda$ and the point $\zeta_\lambda$ belongs to the interval $(a,b)$
which is mapped into $(a+\veps,b-\veps)$ by $g_\lambda.$
\par
Next, observe that for any $x\in(0,1),$
\beq
g_\lambda'(x)=-1+(1+\alpha-\alpha x)e^{-\alpha x}>-1,
\feq
and
\beq
g_\lambda''(x)<0 ~(\mbox{and hence, $g_\lambda'$ is decreasing}),\qquad g_\lambda'(1)=-1+e^{-\alpha}<0.
\feq
Furthermore,
\beq
g_\lambda'(\zeta_\lambda)<1
\feq
because $g_\lambda'(\zeta_\lambda)<0$ when $\zeta_\lambda>\nu_\lambda,$ and when $\zeta_\lambda<\nu_\lambda$ the graph of $g_\lambda$ intersects
the line $y=x$ at $\zeta_\lambda$ going upward from the left to the right, and hence the slope is less than one at the point of intersection.
\par
It follows that we can choose $a,b\in(0,1)$ in such a way that \eqref{abe} holds true for some $\veps>0,$ and
\beqn
\label{der}
\mbox{\rm There exists}~\varrho\in (0,1)~\mbox{\rm such that}~ |g_\lambda'(x)|<\varrho~\mbox{\rm on}~ (a,1).
\feqn
From now on assume that the constants $a,b\in (0,1)$ and $\veps>0$ satisfy \eqref{abe} and \eqref{der}.  Pick any $\delta\in (0,\veps),$
and assume that for some $n\in\nn$ and $k\in\zz_+,$ $\xn_k$ satisfies the following two conditions:
\beqn
\label{cnew}
\begin{array}{lll}
1.&$\mbox{}$\quad& x_{n,k}\in (a+\delta,b-\delta)
\\
2.&$\mbox{}$\quad&|x_{n,k}-\psi_k|< \delta.
\end{array}
\feqn
It follows from \eqref{ker} that for a given $\xn_k<bn,$
\beqn
\label{hoeff}
\nonumber
&&
P\Bigl(\Bigl|x_{n,k+1}-(1-x_{n,k})\Bigl(1-q_n^{\xn_k}\Bigr)\Bigr|\geq \frac{(1-\varrho)\delta}{2}\,\Bigr|\xn_k\Bigl)
\leq 2e^{-\frac{\delta^2 n(1-x_{n,k})}{2(1-\varrho)^2}}
\\
&&
\qquad
\leq 2e^{-\frac{\delta^2 n(1-b)}{2(1-\varrho)^2}},
\feqn
where in the first step we applied Hoeffding's inequality for binomial variables. Furthermore, if $n$ is large enough, then
under the condition \eqref{cnew}, we have
\beq
\Bigl|g_\lambda(x_{n,k})-\bigl(1-x_{n,k}\bigr)\Bigl(1-q_n^{\xn_k}\Bigr)\Bigr|\leq \frac{(1-\varrho)\delta}{2}
\feq
and
\beq
|g_\lambda(x_{n,k})-g_\lambda(\psi_k)|\leq \varrho \delta,
\feq
which together imply
\beq
\Bigl|g_\lambda(\psi_k)-\bigl(1-x_{n,k}\bigr)\Bigl(1-q_n^{\xn_k}\Bigr)\Bigr|\leq \frac{(1-\varrho)\delta}{2}+\varrho \delta
\leq  \frac{(1+\varrho)\delta}{2}.
\feq
Combining the last  inequality with \eqref{hoeff}, we obtain that
\beq
P\Bigl(\bigl|x_{n,k+1}-\psi_{k+1}\bigr|\geq \delta\,\Bigl|\,\xn_k \Bigl)\leq 2e^{-\frac{\delta^2 n(1-b)}{2(1-\varrho)^2}}
\feq
for any $X_{n,k}$ that satisfies condition \eqref{cnew}. Taking in account \eqref{abe}, we arrive to the following conclusion:
\begin{lemma}
\label{uni}
If $\xn_k$ satisfies condition \eqref{cnew}, then (conditionally on $\xn_k$) $\xn_{k+1}$ satisfies the same condition with probability
larger than $2e^{-\frac{\delta^2 n(1-b)}{2(1-\varrho)^2}},$ uniformly on $\xn_{k+1}.$
\end{lemma}
Let $A_{n,m,\delta}$ be the event that \eqref{cnew} is satisfied for $k=0,1,\ldots,m,$ and recall $\tau_n(\delta),$ from \eqref{taun}.
By the Markov property, the lemma implies that under the conditions of the theorem, if \eqref{cnew} is satisfied for $k=0,$ then
\beq
P\bigl(\tau_n(\delta) > m\bigr)\geq P(A_{n,m,\delta})\geq \Bigl(1-2e^{-\frac{\delta^2 n(1-b)}{2(1-\varrho)^2}}\Bigr)^m
\geq 1-2me^{-\frac{\delta^2 n(1-b)}{2(1-\varrho)^2}},
\feq
completing the proof of the theorem.
\end{proof}
The counterpart of Theorem~\ref{da} for $\lambda\leq 1$ is easier to prove since the behavior of the derivative
$g_\lambda'$ on $[0,1]$ is ``more friendly" in this case, in that $|g_'\lambda(x)|<1$ and $g_\lambda(x)<x$ for all
$x\in(0,1).$ For $\veps\in (0,1)$ let
\beqn
\label{taun5}
\varsigma_n(\veps)=\inf\bigl\{k\in\zz_+:x_{n,k}<\veps\bigr\}.
\feqn
Recall $\tau_n(\delta)$ from \eqref{taun}.
Let $\lfloor x \rfloor =\max\{k\in\zz: k\leq x\}$ denote the integer part of the real number $x.$
We have:
\begin{theorem}
\label{oth1}
Let Assumption~\ref{assume7} hold with $\lambda\leq 1,$ and suppose that $x_{n,0}=\psi_0$ for some $\psi_0\in(0,1)$ and all
$n\in\nn.$ Pick any $\delta\in (0,1)$ such that
\beq
\delta<\psi_0-\psi_1,
\feq
where $\psi_k$ are defined in \eqref{psi}.  Let $a\in(0,1)$ be the minimal root of the equation
$g_\lambda(a)=\delta,$ and define
\beqn
\label{vrho}
\varrho=
\left\{
\begin{array}{lll}
\max\{\lambda, |g_\lambda'(\psi_0)|\}&\mbox{\rm if}&\lambda<1,
\\
\max\{g'_\lambda(a), |g_\lambda'(\psi_0)|\}&\mbox{\rm if}&\lambda=1.
\end{array}
\right.
\feqn
Then the following holds true.
\item [(i)] Let $m_0=\bigl\lfloor \log \frac{\psi_0}{\delta}\bigr\rfloor +1.$ Then for any $n\in\nn,$
\beq
P(\varsigma_n(\delta)>m_0)\geq \Bigl(1-2e^{-\frac{\delta^2 n(1-\psi_0)}{2(1-\varrho)^2}}\Bigr)^{m_0}
\geq 1-2m_0\exp\Bigl\{-\frac{\delta^2 n(1-\psi_0)}{2(1-\varrho)^2}\Bigr\}.
\feq
\item [(ii)] If $\lambda\in (0,1),$ then for any $n,m\in\nn,$
\beq
P(\tau_n(\delta)>m)\geq \Bigl(1-2e^{-\frac{\delta^2 n(1-\psi_0)}{2(1-\varrho)^2}}\Bigr)^m
\geq 1-2m\exp\Bigl\{-\frac{\delta^2 n(1-\psi_0)}{2(1-\varrho)^2}\Bigr\}.
\feq
\end{theorem}
\begin{proof}
Observe that \eqref{der} holds with $\varrho$ introduced in \eqref{vrho}.
Furthermore, $\psi_k$ is a decreasing sequence since $g_\lambda(x)<x$ for $x\in (0,1)$ when $\lambda\leq 1.$ The rest of the proof
is similar to that of Theorem~\ref{oth}, namely the induction argument based on \eqref{cnew} and Lemma~\ref{uni} carries over.
Note that we can replace $1-b$ by $1-\psi_0=1-\veps_0$ in the conclusions of the lemma because the sequence $\psi_k$ is monotone decreasing under the conditions
of the theorem.
\end{proof}

\end{document}